\documentclass[12pt,reqno, oneside,english]{amsart}
\usepackage[T1]{fontenc}
\usepackage[latin9]{inputenc}
\usepackage{geometry}
\usepackage{babel}
\usepackage{verbatim}
\usepackage{mathrsfs}
\usepackage{amsbsy}
\usepackage{amstext}
\usepackage{amsthm}
\usepackage{amssymb}
\usepackage[all]{xy}
\PassOptionsToPackage{normalem}{ulem}
\usepackage{ulem}
\usepackage[unicode=true,pdfusetitle,
 bookmarks=true,bookmarksnumbered=false,bookmarksopen=false,
 breaklinks=false,pdfborder={0 0 1},backref=false,colorlinks=false]
 {hyperref}

\usepackage{amsmath}
\usepackage{tikz}
\usepackage{tikz-cd}
\usepackage{booktabs}
\usetikzlibrary{shapes.geometric}
\usepackage{amscd}
\usepackage{mathabx}
\numberwithin{equation}{section}
\usepackage[original]{imakeidx}
\makeindex 
\usepackage{mathrsfs}

\makeatletter
\numberwithin{equation}{section}
\numberwithin{figure}{section}
\theoremstyle{plain}
\newtheorem{thm}{\protect\theoremname}
\theoremstyle{definition}

\theoremstyle{plain}
\newtheorem{cor}[thm]{\protect\corollaryname}
\theoremstyle{remark}
\newtheorem{rem}[thm]{\protect\remarkname}
\theoremstyle{plain}
\newtheorem{lem}[thm]{\protect\lemmaname}
\theoremstyle{plain}
\newtheorem{prop}[thm]{\protect\propositionname}
\theoremstyle{definition}
\newtheorem{example}[thm]{\protect\examplename}
\theoremstyle{remark}
\newtheorem*{claim*}{\protect\claimname}

\usepackage{babel}

\makeatother

\providecommand{\claimname}{Claim}
\providecommand{\corollaryname}{Corollary}
\providecommand{\definitionname}{Definition}
\providecommand{\examplename}{Example}
\providecommand{\lemmaname}{Lemma}
\providecommand{\propositionname}{Proposition}
\providecommand{\remarkname}{Remark}
\providecommand{\theoremname}{Theorem}

\newif\ifdraft\drafttrue
\draftfalse

\newcommand\eq[2]{{\ifdraft{\ \tt [#1]}\else\ignorespaces\fi}\begin{equation}\label{#1}{#2}\end{equation}}
\newcommand {\equ}[1]{\eqref{#1}}

\newcommand{\df}{{\, \stackrel{\mathrm{def}}{=}\, }}

\begin{document}
\title{Random walks on tori and normal numbers in self-similar sets\label{chap:Random-walks}}
\author{Yiftach Dayan, Arijit Ganguly and Barak Weiss}
\begin{abstract}
We study random walks on a $d$-dimensional torus by affine expanding
maps whose linear parts commute. Assuming an irrationality
condition on their translation parts, we prove that the Haar measure
is the unique stationary measure. We deduce that if $K
\subset \mathbb{R}^d$ is an attractor of a finite iterated function
system of $n\geq 2$  
maps of the form $x \mapsto D^{-1} x + t_i \ (i=1, \ldots,
n)$, where
$D$ is an expanding $d\times d$ integer matrix, and is the same for all the maps,  
under an irrationality condition on
the translation parts $t_i$, almost every point in
$K$ (w.r.t. any Bernoulli measure) has an equidistributed
orbit under the map $x\mapsto Dx$ (multiplication mod $\mathbb{Z}^{d}$). In the
one-dimensional case, this conclusion 
amounts to normality to base $D$. Thus for example, almost every point in an
irrational dilation of 
the middle-thirds Cantor set is normal to base 3.
\end{abstract}

\maketitle

\section{Introduction\label{sec:Introduction}}

In this paper we will analyze random walks on a torus. For some random
walks, driven by finitely many expanding affine maps of the torus we will show
that the only stationary measure is the Haar measure, and hence for
any starting point, almost every random trajectory is equidistributed.
We will use the random walks results to obtain new results on
digital expansion of typical points in certain self-affine sets. This paper
follows a scheme similar to that of \cite{SimmonsWeiss2019}, where related results
about random walks on homogeneous spaces were proved, leading to
results on Diophantine properties of typical points on self-similar
sets. 

\subsection{Random walks on tori}
Informally, a random walk on a torus may be described as follows.
Suppose $G$ is a semigroup acting on the torus and $\mu$ is some
probability measure on $G$. Given a point $x$ in the torus, the
random walk proceeds by sampling a random element $g\in G$ according
to $\mu$ and moving the point $x$ to $gx$. The process continues
indefinitely to obtain an infinite random path in the torus.

More formally, let $G$ be a second countable locally compact semigroup
acting on $\mathbb{T}^{d} \df \mathbb{R}^d/\mathbb{Z}^d$ and let $\mu$
be some Borel probability 
measure on $G$. To this system we associate a Bernoulli shift $\left(B,\beta,\mathcal{B},T\right)$,
where $B=G^{\mathbb{N}}$, $\beta=\mu^{\otimes\mathbb{N}}$ is the
product measure on $B$, $\mathcal{B}$ is the Borel $\sigma$-algebra
on $B$ and $T$ is the left shift. For a measure $\nu$ on
$\mathbb{T}^{d}$, the convolution 
of $\mu$ with $\nu$ is the measure on $\mathbb{T}^{d}$ which is
given by
\[
\mu*\nu\left(A\right)=\int\limits _{G}g_{*}\nu\left(A\right)d\mu\left(g\right),
\]
for every measurable set $A\subseteq\mathbb{T}^{d}$. A probability
measure $\nu$ for which $\mu*\nu = \nu$ is called {\em
  $\mu$-stationary.}

Clearly, every $G$-invariant measure is $\mu$-stationary, but the
converse is often false. The action is called {\em stiff} if any
$\mu$-stationary measure is invariant (this terminology was introduced
by Furstenberg in \cite{Furstenberg_stiffness}). Recently there have
been several breakthroughs 
for the case where $G$ acts on $\mathbb{T}^{d}$ by linear automorphisms.
Starting with the work of Bourgain, Furman, Lindenstrauss
and Mozes \cite{bourgain2011stationary}, followed by a series
of papers by Benoist and Quint \cite{BenoistQuint2009mesures,
  BenoistQuint2013stationaryII,BenoistQuint2013stationaryIII}, 
these results gave certain conditions guaranteeing stiffness.

In this
paper we establish stiffness for certain random walks generated
by {\em affine toral endormorphisms}, namely maps $\mathbb{T}^d \to
\mathbb{T}^d$ of the form $x \mapsto D(x) + \alpha$, where $D$ is a
linear toral endomorphism, $\alpha \in \mathbb{T}$ and addition is the
group law on $\mathbb{T}^d$. Moreover we show that these stiff random walks
have a unique invariant measure. 
We recall that a toral endomorphism is a map $\mathbb{T}^d \to \mathbb{T}^d
$ of the form $x \mapsto Dx \  (\mathrm{mod}\  \mathbb{Z}^d)$, where $D$ is a
matrix with integer coefficients.  In this paper we will use the same letter
to denote  a toral
endomorphism and the corresponding integer matrix. We denote the
identity $d \times d$ matrix by $ \mathbb{I}_{d}$. 
We call a  matrix {\em expanding} if all of
its (complex) eigenvalues have modulus greater than 1. 
\begin{thm}
\label{thm:unique stationary measures T^n}Let $D_{1},...,D_{n}$
be commuting $d\times d$ matrices with coefficients in $\mathbb{Z}$.
Assume that all the $D_{i}$ are expanding. For $i=1, \ldots, n$, let
$\alpha_{i}\in\mathbb{R}^{d}$ and let
$$h_i: \mathbb{T}^{d}\to\mathbb{T}^{d}, \ \
h_{i}\left(x\right)=D_{i}(x)+\alpha_{i}
.$$ 
Assume that
\[
\left\{
  \left(\mathbb{I}_{d}-D_{i}\right)\alpha_{j}-\left(\mathbb{I}_{d}-D_{j}\right)
  \alpha_{i}:\,i,j\in\left\{
    1,...,n\right\} \right\}  
\]
is not contained in any proper closed subgroup of
$\mathbb{T}^{d}$. Let $\mu$ be a probability measure such that
$\mathrm{supp} \, \mu = \left\{ h_{1},...,h_{n}\right\} $. Then Haar
measure is the unique $\mu$-stationary 
measure on $\mathbb{T}^{d}$. 
\end{thm}

The uniqueness property of a stationary measure for a random
walk has strong consequences. Indeed, using Breiman's
law of large numbers (\cite{Breiman1960}, see also \cite[Chap. 
2.2]{Benoist-QuintBook}), 
in the setting of Theorem \ref{thm:unique stationary measures T^n},
one obtains that for every $x\in\mathbb{T}^{d}$ and every $\varphi\in
C\left(\mathbb{T}^{d}\right)$, 
for $\beta$ - a.e. $b\in B$, 
\[
{\displaystyle \dfrac{1}{N}\sum_{k=0}^{N-1}\varphi\left(b_{k}\cdots
    b_{1}x\right) \underset{{\scriptstyle
      N\to\infty}}{\longrightarrow} \int\limits
  _{\mathbb{T}^{d}}\varphi \,  d \text{Haar}}.
\]
By a standard argument using the separability of the space
$C\left(\mathbb{T}^{d}\right)$, we obtain that for every starting
point, a.e. trajectory is uniformly distributed. That is: 
\begin{cor}
\label{cor:a.s. equidistribution of trajecotries}For every $x\in\mathbb{T}^{d}$,
for $\beta$-a.e. $b\in B$,\textup{ 
\[
{\displaystyle \dfrac{1}{N}\sum_{k=0}^{N-1}\delta_{b_{k}\cdots b_{1}x}\underset{{\scriptstyle N\to\infty}}{\longrightarrow}\text{Haar}}
\]
} in the weak-{*} topology. 
\end{cor}

\subsection{Normal numbers in self similar sets}

\subsubsection{Iterated function systems and their attractors \label{subsec:Iterated-function-systems}}
In general, an iterated function system (IFS) is a finite collection
of maps $\left\{ \varphi_{i}\right\} _{i\in\Lambda}$ on a complete
metric space $\left(X,d\right)$, where for each $i\in\Lambda$, $\varphi_{i}:X\to X$
is contracting, i.e. for some $\rho\in\left(0,1\right)$, for every
$x,y\in X$, 
\[
d\left(\varphi_{i}x,\varphi_{i}y\right)\leq\rho\cdot d\left(x,y\right).
\]

For every such IFS there is a unique non-empty compact set $K\subseteq\mathbb{R}^{d}$
which satisfies 
\[
K=\bigcup\limits _{i\in\Lambda}\varphi_{i}(K),
\]
and is called the {\em attractor} of the IFS.

In this paper, we specialize to IFSs of $\mathbb{R}^{d}$ that 
consist of affine contractions. That is, for each $i\in\Lambda$,
$\varphi_{i}:\mathbb{R}^{d}\to\mathbb{R}^{d}$ is given by $\varphi_{i}\left(x\right)=A_{i}x+\alpha_{i}$,
where $A_{i}$ is a $d\times d$ matrix, and with respect to some
norm $\left\Vert \cdot\right\Vert $ on $\mathbb{R}^{d}$, for every
$i\in\Lambda$,
\[
\sup_{\begin{array}{c}
 x\in\mathbb{R}^{d}\setminus\{0\}\end{array}}\frac{\left\Vert A_{i}x\right\Vert }{\left\Vert x\right\Vert }<1,
\]
and $\alpha_{i}\in\mathbb{R}^{d}$ is called the {\em translation
of $\varphi_{i}$}. We will refer to an attractor of such an IFS
as a {\em self-affine set}. An important special case is when the
maps $\varphi_{i}$ are all similarity functions, i.e. they are given
by $\varphi_{i}\left(x\right)=r_{i}\cdot O_{i}x+\alpha_{i}$, where
$r_{i}\in\left(0,1\right)$ is called the {\em contraction ratio of $\varphi_{i}$}, $O_{i}$ is an orthogonal map and $\alpha_{i}\in\mathbb{R}^{d}$.
In this case, the attractor is called a \emph{self-similar set}.

Every point in the attractor $K$ of an IFS $\Phi=\left\{ \varphi_{i}\right\} _{i\in\Lambda}$
has a symbolic coding (possibly more than one), given by the so called
coding map $\pi_{\Phi}:\Lambda^{\mathbb{N}}\to K$, which may be defined
by 
\[
\forall i=\left(i_{1},i_{2},...\right)\in\Lambda^{\mathbb{N}},\,\,\pi_{\Phi}\left(i\right)=\lim_{n\rightarrow\infty}\varphi_{i_{1}}\circ\varphi_{i_{2}}\circ\cdots\circ\varphi_{i_{n}}(x_{0}),
\]
where $x_{0}$ is some arbitrary basepoint. It will be convenient
for us to choose $x_{0}=0$. For an introduction see \cite{falconer2013fractal}.

We will say that a measure on $K$ is a {\em Bernoulli measure}
if it is obtained by pushing forward a Bernoulli measure on the symbol
space $\Lambda^{\mathbb{N}}$ by the coding map $\pi_{\Phi}$. In
other words, the measure is of the form $\left(\pi_{\Phi}\right)_{*}P^{\mathbb{\otimes N}}$,
where $P$ is a probability measure on the finite set $\Lambda$ and
$P^{\mathbb{\otimes N}}$ is the product measure on $\Lambda^{\mathbb{N}}$.
Throughout this text, we assume that $P\left(\left\{ i\right\} \right)>0$
for every $i\in\Lambda$ (otherwise we can replace $\Lambda$ with
$\text{supp}\left(P\right)$). This definition depends on the underlying
IFS (rather than its attractor), and thus we shall sometimes refer
to a measure as a {\em $\Phi$-Bernoulli measure}, where $\Phi$
indicates the IFS.

\subsubsection{Normal numbers}
Let $D \geq 2$ be an integer. 
Recall that $x\in\mathbb{R}$ is called {\em normal to base $D$} if for
every $n\in\mathbb{N}$, every finite 
word $\omega\in\left\{ 0,...,D-1\right\} ^{n}$ occurs in the base
$D$ digital expansion of $x$ with asymptotic frequency $D^{-n}$.
Equivalently, $x\in\mathbb{R}$ is normal to base $D$ iff the forward
orbit of $x$ under the map $x\mapsto Dx$ (multiplication by $D$ modulo
1) is equidistributed
w.r.t. Lebesgue measure in $\left[0,1\right]$. A detailed exposition
on normal numbers, and in particular for the equivalence
stated above, may be found in \cite{bugeaud2012distribution}. One
useful property of normal numbers is that a number $x\in\mathbb{R}$
is normal to some base $D$ iff for every $s,t\in\mathbb{Q}$ s.t.
$s\neq0$, $sx+t$ is normal to base $D$ (this property was proved
by Wall in his Ph.D. thesis \cite{Wall1950}).

Since the map $x\mapsto Dx$ is ergodic (w.r.t. Lebesgue
measure on $\left[0,1\right]$), by Birkhoff's ergodic theorem,
a.e. real number is normal to every integer base\footnote{This fact
  was first proved by $\acute{\text{E}}$. Borel in 1909 without 
using ergodic theory.}. Focusing attention to self-similar
sets, one may inquire as to the size of the set of all numbers within
some self similar 
set that are normal to a given base. It was proved in
\cite{BBFKW} that the set 
of real numbers which are not normal to any integer base (these numbers
are called \emph{absolutely non-normal}) intersects any infinite self-similar
$K \subseteq \mathbb{R}$, in a set whose Hausdorff dimension is equal
to that of $K$. This extends
a result of Schmidt \cite{Schmidt1966}, which
provides the same conclusion for $K=\left[0,1\right]$, to nice self-similar
sets.

On the other hand, in many cases, with respect to natural measures
supported on self-similar sets in $\mathbb{R}$, almost every number
is normal to a given base $D$. Of course this is not the case for every
self-similar set and every base. For example, no number in the middle-thirds
Cantor set is normal to base 3. Several positive results in this
direction were obtained in \cite{cassels1959, schmidt1960,
  FeldmanSmorodinsky1992, Host1995, HochmanShmerkin2015}. In all of these
papers, some independence is assumed to hold, between the contraction ratios
of the IFS and the base $D$. In the context of self-similar sets, the
strongest assertion is the following theorem proved in 
\cite{HochmanShmerkin2015}:
\begin{thm}[Hochman-Shmerkin]
\label{thm:Hochman-Shmerkin}Let $K\subseteq\mathbb{R}$ be the attractor
of a contraction IFS $\Phi=\left\{ \varphi_{i}\right\} _{i\in\Lambda}$
with contraction ratios $r_{i}$ (for $i\in\Lambda$), and let $\mu$
be a $\Phi$-Bernoulli measure on $K$. Assume
that $\Phi$ satisfies the open set condition. Then for every integer $D \geq 2$
satisfying
\eq{eq: independence condition}{
  \text{there is } i \in \Lambda \text{  for which }
  \dfrac{\log\left(r_{i}\right)}{\log\left(D\right)}\notin\mathbb{Q},
  }
$\mu$-a.e. number is normal to base $D$.  
\end{thm}

This is a special case of more general results proved in 
\cite{HochmanShmerkin2015}. 
 
\subsubsection{New results}
In the 1-dim case, our result deals with the opposite situation to the one treated
in Theorem \ref{thm:Hochman-Shmerkin}. Instead of assuming that at
least one contraction ratio of the IFS is multiplicatively independent
of the base $D$, we assume that all the contraction ratios of the
IFS are integer powers of $D$. As a substitute for the
`independence' assumption \equ{eq: independence condition}, we impose
an irrationality condition on the translation parts of the functions in the IFS. In its full generality, our result states the following: 
\begin{thm}
\label{thm:normal number is fractals}Let $K$ be the attractor of
an IFS $\Phi\df\{f_{1},...,f_{k}\}$, where for each
$i\in\Lambda\df\left\{ 1,...,k\right\} $, $f_{i}:\mathbb{R}^{d}\to\mathbb{R}^{d}$
is given by
$$f_{i}(x)=D^{-r_i}x+t_{i},$$
for some $t_{i}\in\mathbb{R}^{d}$, $r_i \in \mathbb{N}$ and 
some expanding $d\times d$ matrix $D$ with integer coefficients. Assume that
\eq{eq: irrationality cond1}{
\begin{array}{l}
\text{the set }\left\{ \left(\mathbb{I}_{d}-D^{r_{i}}\right)D^{r_{j}}t_{j}-\left(\mathbb{I}_{d}-D^{r_{j}}\right)D^{r_{i}}t_{i}:\,i,j\in\Lambda \right\} \\
\text{is not contained in any proper closed subgroup of }\mathbb{T}^{d}.
\end{array}}
Then, for every $\Phi$-Bernoulli measure $\mu$ on $K$, for
$\mu$-almost every $x\in K$, the sequence $\{D^{m}x\}_{m=1}^{\infty}$
is equidistributed
in $\mathbb{T}^{d}$ with respect to Haar measure. 
\end{thm}

Note that whenever $D$ is expanding, the functions $f_i$ are indeed contractions
as defined in \S \ref{subsec:Iterated-function-systems}
(see Lemma \ref{lem:expansion implies existence of a norm}).

Observe that in the one-dimensional case, and when $r_1 = \cdots =
r_k=1$, assumption  \equ{eq:
  irrationality cond1} is equivalent to assuming that there exists a pair
$i,j\in\Lambda$ s.t. $t_{i}-t_{j}\notin\mathbb{Q}$, and in case
this condition holds, then w.r.t. any $\Phi$-Bernoulli measure
on $K$, almost every number is normal to base $D$.
In fact, the same conclusion holds for all $\Phi^{n}$-Bernoulli measures,
where $\Phi^{n}\df\left\{ f_{i_{1}}\circ\cdots\circ
  f_{i_{n}}:\,\left(i_{1},...,i_{n}\right)\in\Lambda^{n}\right\} $. 

As an example, denote by $\mathcal{C} \subset [0,1]$ the middle-thirds
Cantor set  and let $K \df \alpha \cdot \mathcal{C} + \beta$,
where $\alpha$ is any irrational number and $\beta \in \mathbb{R}$. Then $K$ is
the attractor of 
$\Phi=\left\{
  x\mapsto\frac{1}{3}x +
  \beta,\,x\mapsto\frac{1}{3}x+\frac{2\alpha}{3} + \beta\right\}
$ 
having a uniform contraction ratio of $3^{-1}$. While prior results
did not provide information regarding normality of typical points in
$K$ to base 3, from Theorem \ref{thm:normal number is fractals}
one may deduce that w.r.t. any $\Phi$-Bernoulli measure on $K$,
almost every point is normal to base 3. Thus Theorem \ref{thm:normal
  number is fractals} 
complements Theorem \ref{thm:Hochman-Shmerkin} of Hochman-Shmerkin,
and combining them we 
have the following corollary.  
\begin{cor}
With the above notations and assumptions, with respect to any
$\Phi$-Bernoulli measure on $K = \alpha \cdot \mathcal{C}+\beta$, 
a.e. point is normal to every integer base $D \geq 2$. 
\end{cor}


\begin{rem}
It would be interesting to extend Theorem \ref{thm:normal number is fractals}
to the more general 
class of attractors of IFSs, in which the maps have different linear parts.
  \end{rem}

In \S\ref{sec:when the condition does not hold} we analyze
the case in which $d=1$ and assumption \equ{eq: irrationality cond1} does not
hold. We also assume that $r_1 = \cdots =r_k=1$. 
In this case we have:
\begin{thm}
  \label{thm:translations are normal to base D.}
Assume that $t_{i}-t_{j}\in\mathbb{Q}$ 
for every $i,j\in\Lambda$, and that $t_{i}$ is normal
to base $D$ for some $i\in\Lambda$. Then for any Bernoulli measure
$\mu$, a.e. $x\in K$
is normal to base $D$. 
\end{thm}

For example, if $\mathcal{C}$ is the Cantor middle-thirds set, and
$\alpha$ is normal to base 3, then so is a.e. $x \in
\mathcal{C}+\alpha$.

\subsection{Acknowledgments} 
We gratefully
acknowledge the support of  BSF grant 2016256 and ISF grant 2095/15. 
The first named author is also grateful for the support of GIF grant 1485-304.6/2019. The second named author would like to thank the Department of Mathematics and Statistics, Indian Institute of Technology Kanpur, for providing an excellent ambience to pursue research. 
\section{Preliminaries\label{sec:Random-walks-on-tori}}
In this section we collect some preliminary results we will need. 

\subsection{Limit measures}
A key ingredient in the proof of Theorem \ref{thm:unique stationary measures T^n}
is the following result (\cite{Furstenberg1963},
see also \cite[Lemmas 1.17, 1.19, 1.21]{Benoist-QuintBook}) which
applies in the setting of random walks as described at the beginning
of section \ref{sec:Introduction}. 
\begin{prop}[Furstenberg]
\label{thm:Furstenberg}Let $\nu$ be a $\mu$-stationary probability
measure on $\mathbb{T}^{d}$, then the following hold: 
\end{prop}

\begin{enumerate}
\item The limit measures $\nu_{b}=\lim\limits
  _{n\to\infty}\left(b_{1}\circ\cdots \circ b_{n}\right)_{*}\nu$
  exist for $\beta$-a.e. $b\in B$.
\item ${\displaystyle \nu=\int \limits _{B}\nu_{b} \, d\beta\left(b\right)}$. 
\item For all $m\in\mathbb{N}$, for
  $\beta\times\mu^{*m}$-a.e. $\left(b,g\right) \in B \times G^m$, we
  have 
$\nu_{b}=\lim\limits _{n\to\infty}\left(b_{1}\circ\cdots\circ b_{n}\circ g\right)_{*}\nu$. 
\end{enumerate}
This is a special case $X = \mathbb{T}^{d}$ of a much more general result. 

\subsection{Commuting expanding matrices}
Recall that a linear transformation (or matrix) is called expanding if all its (complex)
eigenvalues have modulus larger than 1. We shall use the following
characterization of this property. 
\begin{lem}
\label{lem:expansion implies existence of a norm}Let $\mathcal{A}$
be a finite collection of expanding commuting $d\times d$ matrices with entries
in $\mathbb{C}$. Then there exists a norm
$\left\Vert \cdot\right\Vert $ on $\mathbb{C}^{d}$ and some $\rho>1$
such that for every $A\in\mathcal{A}$ and every $x\in\mathbb{C}^{d}$,
$\left\Vert Ax\right\Vert \geq\rho\left\Vert x\right\Vert $. 
\end{lem}

\begin{proof}
Since the matrices commute, there is a basis of $\mathbb{C}^{d}$ with
respect to which they can all be put in an
upper triangular form. Thus we may assume that all the matrices are in
fact upper triangular 
complex matrices. Denote by $\lambda$ the smallest modulus of an
eigenvalue of all the matrices in $\mathcal{A}$, and denote by $a$
the largest modulus of all entries of the matrices. Let $m\in\mathbb{R}$
be any number satisfying 
\[
m>\dfrac{da}{\lambda -1},
\]
and define a norm by 
\[
\left\Vert \left(\begin{array}{c}
x_{1}\\
\vdots\\
x_{d}
\end{array}\right)\right\Vert \df \max_{i}\left\{ m^{i-1}\left|x_{i}\right|\right\} .
\]
A straightforward computation, which we leave to the reader, shows
that for any $x \in \mathbb{C}^d \setminus \{0\}$ and any $A \in
\mathcal{A}$ we have $\|Ax\| >
\|x\|$. Since the unit sphere in $\mathbb{C}^d$ is compact and
$\mathcal{A}$ is finite, the minimum
$$
\rho \df \min\left\{ \frac{\|Ax\|}{\|x\|} : A \in \mathcal{A}, \, x
  \in \mathbb{C}^d, \, \|x\|=1\right\}
$$
exists, is greater than one, and has the required property. 
\end{proof}


\subsection{Invariance of the set of accumulation points of random
  trajectories.} 
The following proposition is a general observation about random walks
on a second countable space, generated by finitely many continuous
maps. It states
that for almost every trajectory, the set of accumulation points along
the trajectory is invariant under each one of the functions. 
\begin{prop}\label{prop: limit sets}
Let $X$ be a second countable space and $f_{j}: X \to X,\
j=1, \ldots, k$
be continuous. Let $\mu \df \sum_j p_{j}\delta_{j},$ 
where $p_{j}>0$ for all $j$ and $\sum p_j =1$, and let
$\mathbb{P}\df \mu^{\bigotimes \mathbb{N}}$. Given $x_{0}\in X$
and $\underline{i}=(i_{1},i_{2},\cdots)\in\{1,\cdots,k\}^{\mathbb{N}}$,
we set $x_{n}(\underline{i})\df f_{i_{n}}\circ\cdots\circ f_{i_{1}}(x_{0})$,
and denote the set of all accumulation
points of $\{x_{n}(\underline{i})\}_{n=1}^{\infty}$ by $L(\underline{i})$.
Then for
$\mathbb{P}$-a.e. $\underline{i}\in\{1,\ldots,k\}^{\mathbb{N}}$ we have 
$f_{j}(L(\underline{i}))\subseteq L(\underline{i})$ for
$j =1, \ldots, k$.
\end{prop}

For the proof of Proposition \ref{prop: limit sets} we will need the following three
lemmas. We retain the same notation  as in the Proposition. We think
of $\{1, \ldots, k\}^{\mathbb{N}}$ as a 
  probability space and use probabilistic notation when discussing the
  $\mathbb{P}$-measure of subsets of this space. 
\begin{lem}
For
$\underline{i}\in \{1,\ldots,k\}^{\mathbb{N}}$,
$N\in\mathbb{N}$ and $\emptyset\neq U\subseteq X$, let
$K^N_U=K_U^N(\underline{i}) \df \{n>N 
:x_{n}(\underline{i})\in U\}.$ Then 
\begin{equation}
\mathbb{P}\left ( 
  \left|K_{U}^{N}\right|=
  \infty\text{ and }\forall n\in K_{U}^{N},i_{n+1}\neq1 
\right)=0.
\end{equation}
\end{lem}

\begin{proof} We will write the elements of $K_{U}^{N}$ in increasing order, as $N<
n_{1}<n_{2}<\cdots$. For any $j\in\mathbb{N}$, let $M_{j}$ denote the
following event:  
\[
\left|K_{U}^{N}\right|\geq j\text{ and }i_{n+1}\neq1,\,\forall n\in\{n_{1},\ldots,n_{j}\}.
\]
That is, $M_j$ is the set of $\underline{i}$ for which, after the $N$-th step,  the sequence
$x_n(\underline{i})$ visits $U$ at least $j$ times, and the next
element of the sequence $\underline{i}$ following each of the first $j$ visits to $U$,
is not equal to 1. 
Observe that $\mathbb{P}(M_{j+1})=\mathbb{P}(M_{j+1}|M_{j})\mathbb{P}(M_{j})$
because $M_{j+1}\subseteq M_{j}$ for any $j$. We have
\[
\begin{split}
\mathbb{P}\left(M_{j+1}\Bigl|\,M_{j}\right)= &
\mathbb{P}\left(\left|K_{U}^{N}\right|\geq j+1\text{ and }\forall
                                               n\in\left\{
                                               n_{1},\ldots,n_{j+1}\right\}
                                               ,i_{n+1}\neq1\,\Bigl|\,M_{j}\right)
                                               \\ =& 
\mathbb{P}\left(\forall n\in\left\{ n_{1},\ldots,n_{j+1}\right\}
                                                     ,i_{n+1}\neq1\,\Bigl|\,\left|K_{U}^{N}\right|\geq
                                                     j+1\text{ and
                                                     }i_{n+1}\neq1,\,\forall
                                                     n\in\{n_{1},\ldots,n_{j}\}\right)\\
                                                   \times \, &  
\mathbb{P}\left(\left|K_{U}^{N}\right|\geq j+1\,\Bigl|\,M_{j}\right).
\end{split}
\]
Since the entry $i_{n_{j+1}+1}$ does not depend on the previous entries,
\[
\begin{split}
& \mathbb{P}\left(i_{n+1}\neq1,\,\forall
  n\in\{n_{1},\ldots,n_{j+1}\}\,\Bigl|\, \left|K_{U}^{N}\right|\geq
  j+1 \text{ and }i_{n+1}\neq1,\,\forall
  n\in\{n_{1},\ldots,n_{j}\}\right)\\ \leq \, & 
\mathbb{P}\left(i_{n_{j+1}+1}\neq1\right)=1-p_{1}.
\end{split}
\]
Hence, by induction, one has $\mathbb{P}(M_{j})\leq(1-p_{1})^{j}$,
for all $j\in\mathbb{N}$. The conclusion of the lemma is now immediate. 
\end{proof}
Denote by $\mathscr{B}$ a countable base of the topology of $X$.
It follows at once that 
\begin{lem}
\label{lem:inv. of accum. points, countable base}$\mathbb{P}\left(\exists B\in\mathscr{B},\exists N\in\mathbb{N}\text{ s.t. }\left|K_{B}^{N}\right|=\infty\text{ and }\forall n\in K_{B}^{N},i_{n+1}\neq1\right)=0$. 
\end{lem}

\begin{lem}
\label{lem:inv. of accum. points}Fix $\underline{i}\in\{1,\ldots,k\}^{\mathbb{N}}$
and $a\in L(\underline{i})$. If $f_{1}(a)\notin L(\underline{i})$
then there exist $B\in\mathscr{B}$ containing $a$, and $N\in\mathbb{N}$ such that
$\left|K_{B}^{N}\right|=\infty\text{ and }\forall n\in
K_{B}^{N},i_{n+1}\neq1$.  
\end{lem}

\begin{proof}
Assume the contrary. That leads to a subsequence
$\{x_{n_{k}}(\underline{i})\}_{k=1}^{\infty}$ 
converging to $a$ such that $i_{n_{k}+1}=1$, for all $k\in\mathbb{N}$.
This shows that $f_{1}(x_{n_{k}})=x_{n_{k}+1}$ for any $k$ and hence,
from the continuity of $f_{1}$, one obtains $x_{n_{k}+1}\rightarrow f(a)$
as $k\rightarrow\infty$. Thus $f_{1}(a)\in L(\underline{i})$ which
contradicts our assumption. 
\end{proof}
\begin{proof}[Proof of Proposition \ref{prop: limit sets}]
Combining Lemmas \ref{lem:inv. of accum. points, countable base}
and \ref{lem:inv. of accum. points} we obtain that
$f_{1}(L(\underline{i}))\subseteq L(\underline{i})$ 
for $\mathbb{P}$-almost all $\underline{i}\in\{1,\ldots,k\}^{\mathbb{N}}$.
By similar arguments, one can prove the same for any $f_{j}$,
$j=1,\ldots,k$.
\end{proof}
\section{Random walks on tori}
In this section we asume the notations and assumptions of Theorem
\ref{thm:unique stationary measures 
  T^n}: we are given $n$ commuting expanding 
$d\times d$ integer matrices $D_{1},...,D_{n}$,
real numbers $\alpha_1, \ldots, \alpha_n$, and define the affine
endomorphisms $h_1, \ldots, h_n$ by $h_i(x) = D_i(x)+ \alpha_i$. We
are also given a probability
measure $\mu$ whose support is the finite set $\left\{
  h_{1},...,h_{n}\right\}$. We set $\beta = \mu^{\otimes \mathbb{N}}$, and  
assume that 
\eq{eq: irrationality cond2}{
\left\{
  \left(\mathbb{I}_{d}-D_{i}\right)\alpha_{j}-\left(\mathbb{I}_{d}-D_{j}\right)\alpha_{i}:
  \, 
  1\leq i,j  \leq n \right\} \text{ is not contained in a proper
closed subgroup of } \mathbb{T}^{d}.
}

We need the following Lemma. 
\begin{lem}
\label{lem:multidim accumulation points}If a finite set
$\{\mathbf{z}_{1},\mathbf{z}_{2},\ldots,\mathbf{z}_{\ell}\} \subseteq
\mathbb{T}^{d}$ is not contained in a proper closed
subgroup of $\mathbb{T}^{d}$, then for $\beta$-a.e.
$\underline{i}=(i_{1},i_{2},\ldots)\in\{1,\ldots,n\}^{\mathbb{N}}$ 
the set of all accumulation points of 
$$\mathscr{S}(\underline{i}) \df\{D_{i_{m}} \circ \cdots \circ 
D_{i_{1}}(\mathbf{z}_{j}) :m\in\mathbb{N},1\leq j\leq\ell\}$$
is not contained in a proper closed subgroup of $\mathbb{T}^{d}$. 
\end{lem}

\begin{proof}
Denote the set of $\underline{i}=(i_{1},i_{2},\cdots)\in\{1,\ldots,n\}^{\mathbb{N}}$
for which the conclusion fails by
$\Lambda_{0}$, and assume by contradiction that $\beta(\Lambda_0)>0$. 
For $\underline{i}\in\Lambda_{0}$, let us denote the closed subgroup
generated by all the accumulation points of $\mathscr{S}(\underline{i})$
by $K(\underline{i})$. Since a torus contains countably many closed
subgroups $K$, we can pass to a subset of $\Lambda_0$ (which we
continue to denote by $\Lambda_0$) such that $K = K(\underline{i})$ is
the same for all $\underline{i}$.
Let $p : \mathbb{R}^d \to \mathbb{T}^d$ be the natural projection
map. Then $p^{-1}(K)$ is a closed subgroup of
$\mathbb{R}^{d}$. Let $S$ be the connected component of the identity
in $p^{-1}(K)$, i.e. the largest subspace contained
in $p^{-1}(K)$. Then $S$ is a rational subspace of $\mathbb{R}^d$,
$\dim S <d$,  and
$p^{-1}(K)/S$ is discrete in $\mathbb{R}^{d}/S$. Consider the
following commutative diagram, where the vertical maps 
are the canonical projection maps and $\bar p$
is the map induced by $p$:
\[
\xymatrix{\mathbb{R}^{d}\ar[r]^{p}\ar[d] & \mathbb{T}^{d}\ar[d]\\
\mathbb{R}^{d}/p^{-1}(K)\ar[r]^{\bar p} & \mathbb{T}^{d}/K
}
\]

\noindent The diagram shows that
$(\mathbb{R}^{d}/S)/(p^{-1}(K)/S)\cong\mathbb{R}^{d}/p^{-1}(K)$ 
is isomorphic to $\mathbb{T}^{d}/K$. Hence $\mathbb{T}^{d}/K$ is
a torus of dimension $d-\dim S$, $\mathbb{R}^d/S$ is the universal cover of
$\mathbb{T}^d/K$,  and the covering map is given by  
\[
\pi:\mathbb{R}^{d}/S\longrightarrow\mathbb{T}^{d}/K,\ \ \ \pi(\mathbf{x}+S)\df
p(\mathbf{x})+K.
\]

Denote by $\mathscr{S}'(\underline{i})$ the set of accumulation
points of $\mathscr{S}(\underline{i})$, and write
$\mathscr{S}'(\underline{i})={\displaystyle
  \bigcup_{j=1}^{k}\mathscr{S}_{j}'(\underline{i})}$, 
where $\mathscr{S}_{j}'(\underline{i})$ is the set of accumulation
points of $\{D_{i_{m}}\cdots
D_{i_{1}}\mathbf{z}_{j}\in\mathbb{T}^{d}:m\in\mathbb{N}\}$, 
for fixed $j \in \{1, \ldots, \ell\}$. It follows from Proposition \ref{prop: limit sets} that
for all $j$, there is a subset of $\Lambda_0$ of full measure, of
$\underline{i}$ for which 
$D_{1}(\mathscr{S}_{j}'(\underline{i})),\ldots,D_{n}(\mathscr{S}_{j}'(\underline{i}))$ 
are all contained in $\mathscr{S}_{j}'(\underline{i})$. We replace $\Lambda_0$
with its subset of full measure for which this holds for all $j$ (and continue to
denote this set by $\Lambda_0$). For 
each $\underline{i} \in \Lambda_0$ and each $r \in \{1, \ldots, n\}$, we have
$D_{r}(K)\subseteq K$. Since $S$ is the
largest subspace contained in $p^{-1}(K)$, this implies
$D_{r}(S)\subseteq S$. Being an expanding map, $D_r$ is invertible and
hence it preserves dimensions of linear spaces, and so we must have
$D_r(S)=S$ (recall that we denote by $D_r$ both a toral endomorphism
and the corresponding linear transformation).  Therefore we may view
$D_{r}$ as inducing a map on  
both $\mathbb{T}^{d}/K$ and $\mathbb{R}^{d}/S$, which we continue to
denote by $D_r$, and we have the
following commutative diagram:
\[
\xymatrix{\mathbb{R}^{d}/S\ar[r]^{ D_{r}}\ar[d]_{\pi} &
  \mathbb{R}^{d}/S\ar[d]^{\pi}
  \\
\mathbb{T}^{d}/K\ar[r]^{D_{r}} & \mathbb{T}^{d}/K 
}
\]
We have that both horizontal maps in this diagram are surjective endomorphisms, 
$D_r$ commutes with the natural projection
map from $\mathbb{T}^{d}\longrightarrow\mathbb{T}^{d}/K$, and the projection
of $\mathscr{S}(\underline{i})$ on $\mathbb{T}^{d}/K$ has $\overline{\mathbf{0}}=K$
as the only accumulation point for all $\underline{i}\in\Lambda_{0}$.

Fix  $\underline{i}=(i_{1},i_{2}\cdots)\in\Lambda_{0}$. 
We first observe that the
projections $\overline {\mathbf{z}}_{j}$ of $\mathbf{z}_{j}$ in
$\mathbb{T}^d/K$ satisfy
$$D_{i_{m}}\ldots
D_{i_{1}}\overline{\mathbf{z}_{j}}\neq\overline{\mathbf{0}}, \ \ \text{
  for some } j \in \{1, \ldots, \ell\} \text{ and 
  all } m\in\mathbb{N}.$$
Otherwise, there
would exist $M\in\mathbb{N}$ for which $D_{i_{M}}\circ \cdots \circ
D_{i_{1}}\overline{\mathbf{z}_{j}}=\overline{\mathbf{0}}$, 
for all $j\in\{1,2,\ldots,\ell\}$. We then consider the closed subgroup
in $\mathbb{T}^{d}/K$ generated by $\{\overline{\mathbf{z}_{i}}:1\leq i\leq\ell\}$.
It is obvious that this subgroup is contained in the kernel of the
surjective endomorphism $D_{i_{M}}\cdots
D_{i_{1}}:\mathbb{T}^{d}\longrightarrow\mathbb{T}^{d}$. 
Since the kernel is finite, so is the subgroup we considered and hence
proper in $\mathbb{T}^{d}/K$. Pulling it back to $\mathbb{T}^{d}$,
one obtains a proper closed subgroup that contains both $K$ and all
$\overline{\mathbf{z}_{j}}$'s. This contradicts our hypothesis. Since
$\overline{\mathbf{0}}$ is the only accumulation point of the sequence
$\{D_{i_{m}}\cdots D_{i_{1}}\overline{\mathbf{z}_{i}}\}_{m=0}^{\infty}$,
it follows from compactness that ${\displaystyle \lim_{m\rightarrow\infty}D_{i_{m}}\cdots D_{i_{1}}\overline{\mathbf{z}_{j}}=\overline{\mathbf{0}}}$
in $\mathbb{T}^{d}/K$, for each $j$. \\

Our next observation is as follows. If we choose a basis of $S$ and
then extend it to a basis $\mathcal{B}$ of $\mathbb{R}^{d}$, then for each $r$, the
matrix representation of $D_{r}$, with respect to this basis, is of the form 
\[
\left(\begin{array}{cc}
D_{S}^{(r)} & \ast \\
0 & D^{(r)}
\end{array}\right).
\]
By hypothesis, all eigenvalues
of $D_{S}^{(r)}$ and $D^{(r)}$ have absolute value $>1$. The matrix
representing $D_r$ with respect to the projection of $\mathcal{B}$ to
$\mathbb{R}^d/S$ is $D^{(r)}$. 
Using Lemma \ref{lem:expansion implies existence of a norm}, there is
a norm $||\cdot||$ on $\mathbb{R}^{d}/S$
such that 
\[
{\displaystyle \rho\df\inf_{1\leq r\leq n,\,\mathbf{v}\in\mathbb{R}^{d}/S,\,||\mathbf{v}||=1}||D_{r}\mathbf{v}||>1.}
\]
We let $||D_{r}||_{\mathrm{op}}$
denote the operator norm of the linear map  $D_{r}$,
in $\mathbb{R}^{d}/S$ with respect to
the norm on $\mathbb{R}^{d}/S$ chosen above, and let  $R\df\max_{1\leq
  r\leq k}||D_{r}||_{\mathrm{op}}$. 

As $\pi$ is a covering homomorphism, we can take a small enough open
ball $\mathfrak{B}$ in $\mathbb{R}^{d}/S$ centered at $\mathbf{0}$
which evenly covers $\pi(\mathfrak{B})$, i.e. 
$\pi|_{\mathfrak{B}}$ is a homeomorphism onto its image. 
Since ${\displaystyle \lim_{m\rightarrow\infty}D_{i_{m}}\cdots
  D_{i_{1}}\overline{\mathbf{z}_{j}}=\overline{\mathbf{0}}}$, 
there is $N\in\mathbb{N}$ such that $D_{i_{m}}\cdots D_{i_{1}}\overline{\mathbf{z}_{j}}$
lies in $\pi\left(\frac{1}{R}\mathfrak{B}\right)$ for all $m\geq N$.
We denote the lift of $D_{i_{m}}\cdots D_{i_{1}}\overline{\mathbf{z}_{j}}$
in $\frac{1}{R}\mathfrak{B}$ by $\mathbf{t}_{m}$, for all $m\geq N$.
Note that $\mathbf{t}_{m}\neq\overline{\mathbf{0}}$ for any $m\geq
N$. Since each application of $D_r$ increases the norm of a nonzero
vector by at least
$\rho$ and at most $R$, there exists $s\in\mathbb{N}$ such that $D_{i_{m+s}}\cdots
D_{i_{m+1}}\mathbf{t}_{N}\in\mathfrak{B}\setminus\frac{1}{R}\mathfrak{B}$. Since
$\pi|_{\mathfrak{B}}$ evenly covers its 
image, we see that $\pi(D_{i_{N+s}}\cdots
D_{i_{N+1}}\mathbf{t}_{N})\in\pi(\mathfrak{B})\setminus\pi\left(\frac{1}{R}
  \mathfrak{B}\right)$, i.e.,
$D_{i_{N+s}}\cdots D_{i_{N+1}}D_{i_{N}}\cdots
D_{i_{1}}\overline{\mathbf{z}_{j}}\notin\pi\left(\frac{1}{R}\mathfrak{B}\right)$. This
contradicts our choice of $N$.
\end{proof}

We are now ready for the 
\begin{proof}[Proof of Theorem \ref{thm:unique stationary measures
    T^n}]
A straightforward induction shows that for any finite sequence
$j_{1},j_{2},\ldots,j_{m}\in\{1,2,\ldots,n\}$, 
\begin{equation}
h_{j_{1}}\circ\,\cdots\circ\,h_{j_{m}}(\mathbf{x})=D_{j_{1}}\circ
\cdots \circ
D_{j_{m}}(\mathbf{x}) +\sum_{s=1}^{m}D_{j_{1}}\circ \cdots \circ D_{j_{s-1}}(\alpha_{j_{s}})
\end{equation}
(with $+$ denoting addition in $\mathbb{T}^d$). 
From this it is clear that, for any
finite sequence $j_{1},\ldots,j_{m}$ and any pair of indices $\ell,s\in\{1,2,\ldots,n\}$,
one has
\noindent 
\begin{equation}
\begin{split}
&
h_{j_{1}}\circ\,\cdots\circ\,h_{j_{m}}\circ\,h_{\ell}\circ\,h_{s}(\mathbf{x})\\=\ &
h_{j_{1}}\circ\,\cdots\circ\,h_{j_{m}}\circ\,h_{s}\circ\,h_{\ell}(\mathbf{x})+D_{j_{1}}\circ
\cdots \circ 
D_{j_{m}}((\mathbb{I}_{d}-D_{\ell})(\alpha_{s})-(\mathbb{I}_{d}-D_{s})(\alpha_{\ell})).
\end{split}\label{above1}
\end{equation}
For a given vector $\mathbf{a}\in\mathbb{T}^{d}$, let
$$R_{\mathbf{a}}: \mathbb{T}^d \to \mathbb{T}^d, \ \ R_{\mathbf{a}}(x)
\df  x+ \mathbf{a}.$$
We also denote, for $m \in \mathbb{N}$ and $\underline{j} \in \{1,
\ldots, n\}^{\mathbb{N}}$, 
\[
\mathbf{a}_{m, \underline{j}}^{\ell,s}\df D_{j_{1}}\circ \cdots\circ
D_{j_{m}}((\mathbb{I}_{d}-D_{\ell})(\alpha_{s})-(\mathbb{I}_{d}-D_{s})
(\alpha_{\ell})).
\]
With these notations we can rewrite  (\ref{above1})  as 
\begin{equation}
h_{j_{1}}\circ\cdots\circ h_{j_{m}}\circ
h_{\ell}\circ\,h_{s}(\mathbf{x})=R_{\mathbf{a}_{m,
    \underline{j}}^{\ell,s}} \circ h_{j_{1}}\circ\cdots\circ
h_{j_{m}}\circ h_{s} \circ\,h_{\ell}(\mathbf{x}).\label{rewritten}
\end{equation}
\\
 Suppose now that $\nu$ is a $\mu$-stationary measure on $\mathbb{T}^{d}$.
From (\ref{rewritten}) and Proposition \ref{thm:Furstenberg}(3) we
obtain that there is a subset $B_0 \subset B$, with $\beta(B_0) =1$,
such that for any $b \in B_0$ and 
any $\ell,s\in\{1,2,\ldots,n\}$, 
\eq{eq: before continuity}{
\nu_{b}=\lim_{m\rightarrow\infty}(b_{1}\circ\,\cdots\circ\,b_{m}\circ
\,h_{\ell}\circ\,h_{s})_{*}\nu=\lim_{m\rightarrow\infty}\left(R_{\mathbf{a}_{m,
  \underline{j}}^{\ell,s}} \right)_{*}
(b_{1}\circ\,\cdots\circ\,b_{m}\circ\,h_{s}\circ h_{\ell})_{*}\nu\,.
}
Here we identify $B$ with $\{1, \ldots, n\}^{\mathbb{N}}$ in the obvious way
and use the same notation $\beta$ for the Bernoulli measure on the symbol
space $\{1, \ldots, n\}^{\mathbb{N}}$, and identify $\underline{j}$
with $b$; this should cause no
confusion.
Let $\mathrm{Prob}(\mathbb{T}^d)$ denote the space of Borel
probability measures on $\mathbb{T}^d$, equipped with the weak-*
topology. 
The addition map $$\mathbb{T}^d \times \mathbb{T}^d \to \mathbb{T}^d,\ \
\ (\mathbf{a}, x) \mapsto R_{\mathbf{a}}(x) 
$$
is continuous,
and thus the induced map
$$\ \mathbb{T}^d \times
\mathrm{Prob}(\mathbb{T}^d) \to \mathrm{Prob}(\mathbb{T}^d),\  \ \
(\mathbf{a}, \theta) \mapsto 
(R_{\mathbf{a}})_* \theta$$
is also continuous. Thus, by passing to subsequences and using
\equ{eq: before continuity}, we find that
for any $b \in B_0$,
for any accumulation point $\mathbf{a}'$ of  the sequence
$\{\mathbf{a}_{m, \underline{j}}^{\ell,s}\}_{m=0}^{\infty}$, 
the measure $\nu_{b}$ is invariant under $R_{\mathbf{a}'}$. 

We now appeal to condition \equ{eq: irrationality cond2} and Lemma
\ref{lem:multidim accumulation points},  which ensure that for
$\beta$-a.e. $\underline{j}\in\{1,\ldots ,n\}^{\mathbb{N}}$,
the accumulation points of the set
$$\mathbf{A}_{\underline{j}} \df \left\{
  \mathbf{a}_{m, \underline{j}}^{\ell,s}:\,\ell,s\in\left\{ 1,...,n\right\}
  ,\,m\in\mathbb{N}\right\}$$
generate a dense subgroup of $\mathbb{T}^{d}$.
Upon possibly  replacing $B_0$ with its subset, we still
have $\beta(B_0) =1$, and for all $\underline{j} \in B_0$ we also have that the
accumulation points of $\mathbf{A}_{\underline{j}}$ generate a dense subgroup of
$\mathbb{T}^d$. We obtain that for $b \in B_0$,
$\nu_{b}$ is invariant under a dense subgroup of $\mathbb{T}^d$, and
since the stabilizer of a measure is a closed subgroup, that
$\nu_b$ is the
Haar measure on $\mathbb{T}^{d}$. The conclusion of Theorem
\ref{thm:unique stationary measures T^n} now follows from Proposition
\ref{thm:Furstenberg}(2). 
\end{proof}

\subsection{Further remarks.}
We first note that in the one-dimensional case, a stronger version of
Lemma \ref{lem:multidim accumulation points} is true:  its conclusion
holds for {\em every} sequence
$\underline{i}=(i_{1},i_{2},\cdots)\in\{1,\cdots,k\}^{\mathbb{N}}$.  
\begin{lem}
Let $\alpha\in\mathbb{T\setminus}\mathbb{Q}$, and let $D_{1},...,D_{n}
\geq 2$ be integers, which we think of also as toral endomorphisms $D_i: \mathbb{T}
\to \mathbb{T}$. Given a sequence $i_{1},i_{2},...\in\left\{
  1,...,n\right\} ^{\mathbb{N}},$ 
denote $x_{k} \df D_{i_{1}} \circ \cdots \circ D_{i_{k}}(\alpha
)$. Then the set of all accumulation points of the sequence
$\left(x_{k}\right)_{k\in\mathbb{N}}$ is infinite.%
\end{lem}
We will not be using this result, and we leave the proof to the
reader. 

The following shows that condition \equ{eq: irrationality cond2}
cannot be relaxed, at least in case $d=1$: 
\begin{prop}\label{prop: when the condition does not hold}
Assume that for some $i_{0}\in\Lambda$, $\left(1-D_{i_{0}}\right)\alpha_{j}-\left(1-D_{j}\right)\alpha_{i_{0}}\in\mathbb{Q}$
for every $j\in\Lambda$. Then there exists a finitely supported $\mu$-stationary
measure on $\mathbb{T}$. 
\end{prop}

\begin{proof}
Without loss of generality assume that $i_{0}=1$. Denote 
\[
\beta_{j}=\alpha_{j}-\dfrac{D_{j}-1}{D_{1}-1}\alpha_{1}.
\]
By assumption $\beta_{j}\in\mathbb{Q}$ for every $j\in\Lambda$.
Let $q\in\mathbb{N}$ be a common denominator for all the $\beta_{i}$,
and denote $A=\left\{ 0,\frac{1}{q},...,\frac{q-1}{q}\right\} $.
Denote also $x_{0}=-\dfrac{\alpha_{1}}{D_{1}-1}$ ,
so that $h_{1}\left(x_{0}\right)=x_{0}$.

We now claim that for all $ i\in\Lambda$, $h_{i}\left(A+x_{0}\right)\subseteq A+x_{0}$.
Indeed, for all $a\in A$ and all $ i\in\Lambda$, 
\[
\begin{array}{ll}
h_{i}\left(a+x_{0}\right) & =D_{i}(a)+\alpha_{i}+D_{i}(x_{0})\\
 & =D_{i}(a)+\alpha_{i}+\left(D_{i}-1\right)(x_{0})+x_{0}\\
 & =D_{i}(a)+\alpha_{i}-\dfrac{D_{i}-1}{D_{1}-1}\alpha_{1}+x_{0}\\
 & =D_{i}(a)+\beta_{i}+x_{0}\in A+x_{0}
\end{array}
\]
Hence $A+x_{0}$ supports a $\mu$-stationary measure. 
\end{proof}

\begin{rem}
It may be verified that the hypothesis of Proposition \ref{prop: when the condition does not hold} 
is equivalent to the seemingly stronger assumption that for every $i, j\in\Lambda$, 
$\left(1-D_{i}\right)\alpha_{j}-\left(1-D_{j}\right)\alpha_{i}\in\mathbb{Q}$, 
which is the exact converse of condition \equ{eq: irrationality cond2} in the special case treated here.
This equivalence is also implied by combining Proposition \ref{prop: when the condition does not hold} 
and Theorem \ref{thm:unique stationary measures T^n}.

\end{rem}

Our method of proof also gives another proof of the following
well-known fact: 
\begin{prop}\label{prop: CD}
  Suppose $D_{i}=1$ for every $i\in\Lambda$, that is, each $h_i$ is a
  translation by $\alpha_i$. Then Haar measure
is the unique $\mu$-stationary probability measure on $\mathbb{T}$
if and only if $\alpha_i$ is irrational for some  $i\in\Lambda$.
\end{prop}

\begin{proof}
Assume first (without loss of generality) that $\alpha_{1}\notin\mathbb{Q}$.
Since the functions $h_{i}$ are now only rotations, they
commute with each other. Hence, if $\nu$ is some $\mu$-stationary
measure, by Proposition \ref{thm:Furstenberg}, for $\beta$-a.e. $b\in B$,
\[
\begin{array}{l}
\nu_{b}=\lim\limits _{k\to\infty}\left(b_{1}\circ\cdots\circ b_{k}\circ h_{1}\right)_{*}\nu=\lim\limits _{k\to\infty}\left(h_{1}\right)_{*}\left(b_{1}\circ\cdots\circ b_{k}\right)_{*}\nu=\left(h_{1}\right)_{*}\nu_{b}.\end{array}
\]
Since $h_{1}$ is an irrational rotation, $\nu_{b}$ has to be Haar
measure and hence $\nu$ is Haar measure.

The other implication is trivial. 
\end{proof}

We remark that Proposition \ref{prop: CD} could also be obtained as a
corollary of a result of Choquet and Deny (see \cite[\S1.5]{Benoist-QuintBook} for the
statement and a similar
argument). 
\section{Normal numbers in fractals \label{sec:Application}}

Using an idea of \cite{SimmonsWeiss2019}, we will show how to derive
information about orbits in self-affine sets from random walks on tori.
In particular, we will obtain information on normal numbers in self-similar sets. 
We will then use the results of \S\ref{sec:Random-walks-on-tori}
to prove Theorem \ref{thm:normal number is fractals}. In order to make
the idea more transparent, we first prove a special case of Theorem
\ref{thm:normal number is fractals}, namely we assume $r_1 = \cdots =
r_k=1$, or in other words,
\eq{eq: extra assumption}{
  \text{ there is an expanding integer matrix } D 
 \text{ such that } f_i (x ) =D^{-1} x + t_i, \ \ i= 1, \ldots, k.}
In this case the irrationality assumption \equ{eq: irrationality cond1} simplifies to 
\eq{eq: irrationality cond3}{
\{t_i - t_j : 1 \leq i, j \leq k \} \text{ is not contained in any
  proper  closed subgroup of } \mathbb{T}^d.
 }

We will need the following result\footnote{We use here and further
  below results of \cite[\S 5]{SimmonsWeiss2019}, where 
  $\Gamma$ is taken to be a group rather than a semigroup. However the
  more general case of semigroups follows from the same proof.}:
\begin{prop}[\cite{SimmonsWeiss2019}, Prop. 5.1]
  \label{prop:Simmons-Weiss-equi-of-pairs}
Let $\Gamma$ be a semigroup acting on a space $X$, let $\mu$ be a
probability measure on $\Gamma$, and denote the infinite product
measure $\mu^{\otimes \mathbb{N}}$ by $\beta$. 
  Given any $x_{0}\in X$,
assume that for $\beta$-a.e. $b\in B$, the random path
$\left(b_{n}\cdots b_{1}x_{0}\right)_{n\in\mathbb{N}}$ 
is equidistributed w.r.t. a measure $\nu$ on $X$. Then
for $\beta$-a.e. $b\in B$, the sequence 
\[
\left(b_{n}\cdots b_{1}x_{0}, \, T^{n}b\right)_{n\in\mathbb{N}}
\]
is equidistributed w.r.t. the product measure $\nu\otimes\beta$ on
$X\times B$.  
\end{prop}

\begin{proof}[Proof of Theorem \ref{thm:normal number is fractals}
  assuming \equ{eq: extra assumption}]
Let $\mu$ be a Bernoulli measure on $K$ given by
$\mu=\pi_{*}\, \sigma$, where $\sigma$ is a Bernoulli measure
on the symbolic space $\Lambda^{\mathbb{N}}$ and $\pi:
\Lambda^{\mathbb{N}} \to K$ is the coding map. 

By a routine induction, we find that for any $n,m\in\mathbb{N}$
with $n<m$, 
\[\begin{split}
& f_{i_{1}}\circ f_{i_{2}}\circ\cdots\circ
f_{i_{m}}(0)=  D^{-1}f_{i_{2}}\circ\cdots\circ
  f_{i_{m}}(0)+t_{i_{1}}  = \cdots \\ \cdots = &   D^{-n}f_{i_{n+1}}\circ\cdots\circ
  f_{i_{m}}(0) + \sum_{j=1}^{n}D^{-(j-1)}(t_{i_{j}}).
\end{split}
\]
Multiplying by $D^n$ and taking the limit as $m \to \infty$ we obtain
that for any $x\in K$ and 
$n\in\mathbb{N}$, 
\begin{equation}
D^{n}(x)=\sum_{j=1}^{n}D^{n-(j-1)}(t_{i_{j}})+\pi(T^{n}(\underline{i})),\label{eq: D^nx}
\end{equation}
where $x=\pi(\underline{i})$ for some
$\underline{i}=(i_{1},i_{2},...)\in\Lambda^{\mathbb{N}}$.  

Define, for each $s\in\Lambda$, $h_{s}:\mathbb{T}^{d}\to\mathbb{T}^{d}$
by $h_{s}(x)\df D(x+t_{s})$.
Note that
\begin{equation}
h_{i_{n}}\circ\cdots\circ h_{i_{1}}(0)=\sum_{j=1}^{n}D^{n-(j-1)}(t_{i_{j}}),\label{eq:h_i(0)}
\end{equation}
and that $\left(h_{i_{n}}\circ\cdots\circ h_{i_{1}}(0)\right)_{n\in\mathbb{N}}$
is in fact a random walk trajectory governed by the probability measure
$\sigma$ on $\Lambda^{\mathbb{N}}$. Condition \equ{eq:
  irrationality cond1} in this case implies
condition \equ{eq: irrationality cond2} of Theorem \ref{thm:unique
  stationary measures T^n}. Denote by $\lambda$ the Haar measure on $\mathbb{T}^d$. Applying Corollary 
\ref{cor:a.s. equidistribution of trajecotries}, 
we get that for $\sigma$-a.e. $\underline{i}\in\Lambda^{\mathbb{N}}$,
the sequence \eqref{eq:h_i(0)} 
equidistributes
in $\mathbb{T}^{d}$ w.r.t. to $\lambda$. Next, we apply Proposition
\ref{prop:Simmons-Weiss-equi-of-pairs}, with $X = \mathbb{T}^d$ and
$\nu = \lambda$, and obtain the equidistribution
of the sequence $\left(h_{i_{n}}\circ\cdots\circ
  h_{i_{1}}(0),\,T^{n}(\underline{i}) \right)_{n=1}^{\infty}$
w.r.t. the product measure $\lambda \otimes\sigma$ on $\mathbb{T}^d
\times B$, for $\sigma$-a.e.
$\underline{i}\in\Lambda^{\mathbb{N}}$. Since the coding map $\pi$
is continuous, this implies that for
$\sigma$-a.e. $\underline{i}\in\Lambda^{\mathbb{N}}$, 
the joint sequence 
\begin{equation}
\left(h_{i_{n}}\circ\cdots\circ
  h_{i_{1}}(0),\,\pi(T^{n}(\underline{i}))
\right)_{n=1}^{\infty} \label{eq: Pair to equidistribute} 
\end{equation}
is equidistributed in $\mathbb{T}^{d}\times K$ with respect to the
product measure $\lambda\otimes\mu$.

Consider the addition map
$$F:\mathbb{T}^{d}\times K\longrightarrow\mathbb{T}^{d}, \ \ 
F(x,y)=x+y.$$ It follows easily from the fact that Haar measure is
invariant under addition in $\mathbb{T}^d$, that
\begin{equation}\label{eq: from this}
  F_{*}(\lambda\otimes \mu)=\lambda.
  \end{equation}
The equidistribution of $\left(D^{n}x\right)_{n=1}^{\infty}$
for $\mu$-almost every $x\in K$ follows immediately from \eqref{eq:
  from this},
from (\ref{eq: D^nx}), (\ref{eq:h_i(0)}) and from the equidistribution
of (\ref{eq: Pair to equidistribute}).
\end{proof}

For the general case of Theorem \ref{thm:normal number is fractals},
we will need an extension of Proposition
\ref{prop:Simmons-Weiss-equi-of-pairs}. We let $C_r \df \mathbb{Z} /
r\mathbb{Z}$ denote the cyclic group of order $r$, and let $\theta$
denote the uniform measure on 
  $C_r$. Recall that we have a random walk on $\mathbb{T}^d$ driven by
  a finitely supported measure $\mu$ on a semigroup $\Gamma$ of affine toral
  endomorphisms. We write $\bar B \df (\mathrm{supp} \mu)^{\otimes
    \mathbb{Z}}$ and $\bar \beta \df \mu^{\otimes \mathbb{Z}}$. Note this is the
  two-sided shift space. We continue to denote by $T$ the shift map
  (but this time on $\bar B$).  We say that a map
  $\kappa: \Gamma \to C_r$ is a {\em 
    morphism} if it satisfies
  $$
\forall \gamma_1, \gamma_2 \in
  \Gamma, \ \kappa(\gamma_1 \gamma_2) = \kappa (\gamma_1) +
  \kappa(\gamma_2). 
  $$
 
  \begin{prop}\label{prop: independence 2} With the notation above,
let $r \in \mathbb{N}$ and let $\kappa
  : \Gamma \to C_r$ be a surjective morphism. 
Given any $x_{0}\in\mathbb{T}^{d}$,
assume that for $\bar \beta$-a.e. $b\in \bar B$, the random path
$\left(b_{n}\cdots b_{1}x_{0}\right)_{n\in\mathbb{N}}$ 
is equidistributed w.r.t. a measure $\nu$ on $\mathbb{T}^{d}$. Then
for $\bar \beta$-a.e. $b\in \bar B$, the sequence 
\[
\left(b_{n}\cdots b_{1}x_{0}, \, \kappa(b_n \cdots b_1), \, T^{n}b \right)_{n\in\mathbb{N}}
\]
is equidistributed w.r.t. the product measure $\nu\otimes
\theta\otimes  \bar \beta$ on
$\mathbb{T}^{d} \times C_r \times \bar B$.  
\end{prop}

\begin{proof}
This is not stated explicitly in \cite{SimmonsWeiss2019} but is proved
along the same lines as \cite[Thm. 2.2]{SimmonsWeiss2019}. For completeness we
sketch the proof.

An extension of Proposition \ref{prop:Simmons-Weiss-equi-of-pairs} to
the bi-infinite sequence space $\bar B$ is given in
\cite[Prop. 5.2]{SimmonsWeiss2019}. Putting this result to use, it is enough to
show that for any $x_0 \in \mathbb{T}^d$, for  $\bar \beta$-a.e. $b$,
the sequence $\left(b_n \cdots b_1 x_0 , \kappa(b_n \cdots
  b_1)\right)$ is equidistributed, with
respect to $  \nu \otimes\theta$ on $\mathbb{T}^d \times C_r$. Indeed,
once this is established we can apply
\cite[Prop. 5.2]{SimmonsWeiss2019}  with 
$X = \mathbb{T}^d \times C_r$.

Now note that the action of $\Gamma$ on
$\mathbb{T}^d \times C_r$ by $\gamma(x, s) \mapsto (\gamma x, s+
\kappa(\gamma))$ is ergodic. This follows from the fact that each
individual element of $\gamma$ is an expansive toral endomorphism, and
hence is mixing. Thus the required equidistribution statement follows
from \cite[Cor. 5.5]{SimmonsWeiss2019}.
\end{proof}

We will also need the following lemma:
\begin{lem}\label{lem: equidistribution of powers}
Let $(X, T, \mu, \mathscr{B})$ be a probability preserving dynamical system, where $X$ is compact, $\mathscr{B}$ is the Borel $\sigma$-algebra, and $T$ is continuous.
Assume that for some point $x\in X$, and some integer $p>0$,  the sequence $(T^{pn}x)_{n=0}^{\infty}$ is equidistributed w.r.t. $\mu$, then so is the sequence  $(T^{n}x)_{n=0}^{\infty}$.
\end{lem}

\begin{proof}
Given a continuous function $f\in C(X)$, we have
$$\frac{1}{N} \sum_{n=0}^{N-1}{f(T^{pn}x)} \underset{N\to\infty}{\longrightarrow}\int{fd\mu} .$$
For every integer $0\leq m<p$, $f\circ T^m$ is also continuous, and therefore
$$\frac{1}{N} \sum_{n=0}^{N-1}{f(T^{pn+m}x)} \underset{N\to\infty}{\longrightarrow}\int{f\circ T^m d\mu}=\int{fd\mu} .$$
Summing over all possible values of $m$, we get
$$\frac{1}{N} \sum_{n=0}^{N-1}{f(T^{pn}x)+f(T^{pn+1}x)+\cdots + f(T^{pn+(p-1)}x)} \underset{N\to\infty}{\longrightarrow} p\int{fd\mu} .$$
Since the left hand side is just $\frac{1}{N} \sum\limits_{n=0}^{pN-1}{f(T^{n}x)}$, dividing both sides by $p$, we get
$$\frac{1}{pN} \sum_{n=0}^{pN-1}{f(T^{n}x)} \underset{N\to\infty}{\longrightarrow}\int{fd\mu} .$$
Since $f$ is bounded, the claim follows.
\end{proof}

\begin{proof}[Proof of Theorem \ref{thm:normal number is fractals},
  general case]
  We have a set $K$ which is the attractor of the maps
  $f_i (x) = D^{-r_i} \cdot x + t_i \ (i=1, \ldots, k),$ for
  an expanding $d\times d$ integer matrix $D$. We
  can assume with no loss of generality that
  \eq{eq: gcd}{\gcd(r_1, \ldots, r_k) = 1;}
  indeed, if this does  not hold, we can replace $D$ with a power of
  $D$ and (using Lemma \ref{lem: equidistribution of powers}) reduce to the situation in
which \equ{eq: gcd} holds. Write
$$D_i \df  D^{r_i}, \ \ r \df \max r_i, \ \text{ and }\bar D \df D^r.$$
Again by Lemma \ref{lem: equidistribution of powers}, it suffices to show that
the $\bar{D}$-orbit of a.e. $x \in K$ is equidistributed. Computing as in \eqref{eq: D^nx}
and \eqref{eq:h_i(0)} we see that for any $b = (i_1, \ldots) \in B$ and $n \in
\mathbb{N}$, 
\eq{eq: first}{
D_{i_1} \circ \cdots \circ D_{i_n} \pi(b) = \pi(T^nb) + h_{i_n} \circ
\cdots \circ h_{i_1}(0). 
}
For each $n$ and $b$, let $\ell =\ell_{b,n} \in \mathbb{N}$ satisfy
$r(\ell-1) <r_{i_1}+\cdots+r_{i_n} \leq r\ell$, and let
$s = s_{b,n}$ so that
\eq{eq: second}{
\bar{D}^\ell = D_{i_1} \cdots D_{i_n} D^s. 
}
That is,
\eq{eq: cei}{
\ell \df \left \lceil \frac{r_{i_1}+ \cdots + r_{i_n}}{r} \right \rceil
\ \ \text{ and } \ s \df r\ell - \left(r_{i_1} + \cdots + r_{i_n}
\right) \in \{0, \ldots, r-1\}.
}
If we consider $s$ as an element of $C_r$ (i.e. consider its class
modulo $r$) then we see that $s =
\kappa(b_n \cdots b_1)$ where $\kappa$ is the 
morphism mapping $i_j$ to $-r_j$, considered as an element of $
C_r$. The morphism $\kappa$ is surjective in view of \equ{eq: gcd}.

By Corollary \ref{cor:a.s. equidistribution of trajecotries} and
Proposition \ref{prop: independence 2}, the sequence 
$$
\left( h_{i_n} \circ \cdots \circ h_{i_1}(0) , \, s_{b,n}, \,  T^nb\right) \in
\mathbb{T}^d \times C_r \times \bar B
$$
is equidistributed with respect to $\mathrm{Haar} \otimes \theta
\otimes \bar \beta$ for
$\bar \beta$-a.e. $b$.

For $b \in \bar B$ we continue to use the notation $\pi$ to denote the
map $\pi(b) = \pi(i_1, i_2, \ldots)$ where $b = ( \ldots, i_{-1},
i_0, i_1, \ldots) $.
Let $x_0 = \pi(b)$, where $b$ belongs to the subset of $\bar B$ of full
measure for which this equidistribution result holds, we have from
\equ{eq: first} and \equ{eq: second} that 
\eq{eq: two sides}{
\bar{D}^{\ell_{b,n}} (x_0)  = D^{s_{b,n}} \left(\pi(T^nb) + h_{i_n} \circ \cdots \circ h_{i_1}(0) \right).
}
We consider this sequence as a sequence depending on the index $n$,
and note that it is the image of an equidistributed sequence under the continuous map
$$\Psi: \mathbb{T}^d \times C_r \times \bar B \to \mathbb{T}^d, \ \ \Psi (x, c, b) \df
D^c( \pi(b) + x).$$
Thus it is equidistributed with respect to $\lambda$,
 as a sequence of the parameter $n$.

However, our objective is to show equidistribution of the sequence $(\bar{D}^z(x_0))_{z=1}^{\infty}$.
Note that given $b$ as above, the sequence $(\ell_{b,n})_{n\in\mathbb{N}}$ is monotonically increasing, but might have repetitions. To handle these repetitions we introduce the following notation:
for fixed $b$, and for each $n \in\mathbb{N}$, let
$$t_{b} (n)\df \left| \left\{m \in \mathbb{N} : \ell_{b,n} =
    \ell_{b,m} \right\} \right|^{-1} .$$
It is clear from \equ{eq: cei} and the definition of $r$ that
 $t_b(n)\in \left\{\frac{1}{r},\frac{2}{r},\dots,1 \right\}$. 
It is not hard to see that we can compute
$ t_b(n)$ from $s_{b,n}$ and from the symbols $\left(T^n
  b\right)_j$ where $|j| \leq 
  r$. That is, there is a continuous function $\hat{t}$ on $C_r \times
  \bar B$ such that ${ t}_b(n)= \hat{t}(s_{b,n},
  T^nb)$. 
  Now given $f \in C(\mathbb{T}^d)$, we define
  $$
F: \mathbb{T}^d \times C_r \times \bar B \ \ \text{ by } \ F(x, c, b)
\df \left( f \circ \Psi \right) \cdot \hat{t}(c, b) 
$$
so that
$$
F(x, s_{b,n}, T^nb) = \frac{f\left(D^s(\pi(T^n b)+x) \right)}{|\{m\in \mathbb{N}:
  \ell_{b,n} = \ell_{b,m}\}|}.
$$
This definition ensures that the Birkhoff sums of the two sides of
\equ{eq: two sides} satisfy
$$
\sum_{z = 0}^{L-1} f\left(\bar D^{z} x_0 \right) = \sum_{n =0}^{N_L-1}
F\left(h_{i_n} \circ \cdots \circ h_{i_1}(0)  , s_{b,n}, T^nb \right) + O(1),
$$
where $N_L=|\{m\in\mathbb{N}:\ell_{b,m}<L\}|$. Thus
\eq{eq: thus}{
\frac{1}{N_L}
\sum_{z = 0}^{L-1} f\left(\bar D^{z} x_0 \right) \underset{L \to \infty}{\longrightarrow} \int_{\mathbb{T}^d \times C_r \times \bar B} F
\, d( \mathrm{\lambda} \otimes \theta \otimes \bar{\beta}).
}
Applying this with the constant function $f \equiv 1 $ gives the
existence of the limit
$$\xi \df  \int \hat{t} \, d\theta \otimes d \bar \beta =\lim_{L
  \to \infty} \frac{L}{N_L} .$$
Dividing both sides of \equ{eq: thus} by $\xi$ and using Fubini to
compute the right hand side, we see that
$$
\frac{1}{L}
\sum_{z = 0}^{L-1} f\left(\bar D^{z} x_0 \right) \underset{L \to \infty}{\longrightarrow} \int_{\mathbb{T}^d} f 
\, d \mathrm{\lambda} ,
$$
as required.
  \end{proof}
\section{\label{sec:when the condition does not hold}When all the
  differences 
are rational.}

In this section we shall analyze the situation in which condition
\equ{eq: irrationality cond1} in Theorem \ref{thm:normal number is fractals}
does not hold. We focus on the one-dimensional case where $r_{1}=\cdots=r_{k}=1$,
thus we assume throughout this section that all the differences $t_{i}-t_{j}$
are rational. For all $i\in\Lambda$, denote $\delta_{i}\df t_{i}-t_{1}\in\mathbb{Q}$.
Note that $\delta_{1}=0$. Given $\underline{i}\in\Lambda^{\mathbb{N}}$,
recall that 
\[
\pi\left(\underline{i}\right)=\lim_{n\to\infty}f_{i_{1}}\circ\cdots\circ f_{i_{n}}\left(0\right)\in K.
\]
By equation (\ref{eq: D^nx}), for every $m\in\mathbb{N}$, 
\[
D^{m}\left(\pi\left(\underline{i}\right)\right)=\sum_{j=1}^{m}D^{j}\left(t_{1}\right)+\sum_{j=1}^{m}D^{m-(j-1)}\left(\delta_{i_{j}}\right)+\pi(T^{m}(\underline{i}))\,(\text{mod } 1 ).
\]
Denote 
\[
\alpha_{m}\df%
\sum_{j=1}^{m}D^{j}\left(t_{1}\right)%
,\ \ \ \eta_{m}\left(\underline{i}\right)\df%
\sum_{j=1}^{m}D^{m-(j-1)}\left(\delta_{i_{j}}\right)
\]
so that 
\[
D^{m}\left(\pi\left(\underline{i}\right)\right)=\alpha_{m}+\eta_{m}\left(\underline{i}\right)+\pi(T^{m}(\underline{i}))\,(\text{mod } 1 ).
\]
Note that $\eta_{m}\left(\underline{i}\right)$ stays inside the finite
set $\Delta=\left\{ 0,\frac{1}{q},...,\frac{q-1}{q}\right\} $, where
$q$ is a common denominator for $\delta_{2},\delta_{3},...,\delta_{k}$.
Also note that $\alpha_{m}$ is a deterministic sequence (does not
depend on $\underline{i}$), and 

\[
\begin{array}{cl}
\alpha_{m} & =\sum_{j=1}^{m}D^{j}\left(t_{1}\right)\\
 & =\dfrac{D^{m+1}-D}{D-1}t_{1}\,\text{(mod 1)}\\
 & =D^{m}\dfrac{D}{D-1}\,t_{1}-\dfrac{D}{D-1}\,t_{1}\,\text{(mod 1)}.
\end{array}
\]

Recall that $\underline{i}$ is sampled according to some Bernoulli
measure on the space of symbols $\Lambda^{\mathbb{N}}$. In what follows,
we analyze $\eta_{m}\left(\underline{i}\right)$ as a Markov process
with a finite state space. Therefore, we now recall some basic properties
of such processes. A good reference for this topic is \cite{LevinPeresWilmer2009}. 

Recall that a Markov process with a finite state space $\Omega$ is
characterized by a transition matrix $P:\Omega\times\Omega\to\left[0,1\right]$,
where $P\left(\omega_{1,}\omega_{2}\right)$ indicates the probability
of moving from the state $\omega_{1}\in\Omega$ to the state $\omega_{2}\in\Omega$
in one step. The process is called \emph{irreducible} if for every
two states $\omega_{1},\,\omega_{2}\in\Omega$, there exists some
integer $n>0$, such that $P^{n}\left(\omega_{1},\omega_{2}\right)>0.$
An irreducible Markov chain is called \emph{aperiodic} if for some
(equivalently, every) $\omega\in\Omega$, 
\[
\text{g.c.d}\left\{ n>0:\,P^{n}\left(\omega,\omega\right)>0\right\} =1.
\]
It is well known that an irreducible Markov chain with a finite state
space has a unique stationary measure, and in case the process is
aperiodic as well, the distribution of the process at time $n$ converges,
as $n\to\infty$, to the unique stationary measure, regardless of
the starting point, and at an exponential rate (see e.g. \cite[Theorem 4.9]{LevinPeresWilmer2009}).
\begin{lem}
\label{lem: markov} $\eta_{m}\left(\underline{i}\right)$ is an aperiodic,
irreducible Markov process with a finite state space. 
\end{lem}

\begin{proof}
For each $j$ denote $\tilde{\delta}_{j}\df D\cdot\delta_{j}\in\Delta$.
Then 
\[
\eta_{m}\left(\underline{i}\right)=\sum_{j=1}^{m}D^{m-j}\left(\tilde{\delta}_{i_{j}}\right).
\]
This process may be defined by the following recurrence relation:
\[
\begin{array}{l}
\eta_{1}=\tilde{\delta}_{i_{1}}\\
\forall m\geq1,\,\eta_{m+1}=D\cdot\eta_{m}+\tilde{\delta}_{i_{m+1}}
\end{array}
\]

For every $y\in\Delta$, define the map $T_{y}:\Delta\to\Delta$,
by 
\[
T_{y}\left(x\right)=D\left(x\right)+y.
\]
Note that $T_{0}\left(x\right)=D\left(x\right)$, and the recurrence
relation may be written as 
\[
\eta_{m+1}=T_{\tilde{\delta}_{i_{m+1}}}\left(\eta_{m}\right).
\]
Denote by $\Gamma$ the semigroup generated by the functions $\left\{ T_{\tilde{\delta}_{0}},\,T_{\tilde{\delta}_{1}},\,\dots,\,T_{\tilde{\delta}_{k}}\right\} $,
and let $\tilde{\Delta}\subseteq\Delta$ be defined as $\tilde{\Delta}=\left\{ F\left(0\right):\,F\in\Gamma\right\} .$
Note that $\eta_{1}=T_{\tilde{\delta}_{i_{1}}}\left(0\right)$, and
so 
\[
\tilde{\Delta}=\left\{ a\in\Delta:\,\exists m\in\mathbb{N},\,\mathbb{P}\left(\eta_{m}=a\right)>0\right\} .
\]
Since the variables $\left(\tilde{\delta}_{i_{j}}\right)_{j=1}^{\infty}$
are IID, this is indeed a Markov process on the finite state space
$\tilde{\Delta}$. In order to show that it is irreducible, it is
enough to show that for every $F\in\Gamma$, $\exists G\in\Gamma$
s.t. $G\circ F\left(0\right)=0$. Given such $F\in\Gamma$, denote
$x=F\left(0\right)$. Note that for every $\alpha\in\Delta$, $F\left(\alpha\right)=D^{m}\left(\alpha\right)+x$
for some constant $m\in\mathbb{N}$. Since $\tilde{\Delta}$ is finite,
the orbit $\left\{ D^{l}\left(x\right):\,l\in\mathbb{N}\right\} $
is eventually periodic, and so there exist some $l,\,s$ such that
$D^{l}\left(x\right)=D^{l+s}\left(x\right)$. We may assume that $s\geq m$,
otherwise we may replace $s$ by a large enough positive integer multiple
of $s$. We denote $y=D^{l}\left(x\right)$, so that $D^{s}\left(y\right)=y$.
Now, we note that for every $w\in\mathbb{N}$, 
\[
\begin{array}{cl}
w\cdot y\,\text{(mod 1)} & =D^{s\left(w-1\right)}\left(y\right)+D^{s\left(w-2\right)}\left(y\right)+\cdots+y\\
 & =D^{l}\left(D^{s\left(w-1\right)}\left(x\right)+D^{s\left(w-2\right)}\left(x\right)+\cdots+x\right)\\
 & =T_{0}^{l}\circ\left(F\circ T_{0}^{s-m}\right)^{w-1}\left(x\right)
\end{array}
\]
And so, denoting $G=T_{0}^{l}\circ\left(F\circ T_{0}^{s-m}\right)^{w-1}$,
we have found $G\in\Gamma$ such that $G\left(x\right)=w\cdot y\,\text{(mod 1)}$.
In particular, we may find $G\in\Gamma$ such that $G\left(x\right)=q\cdot y\,\text{(mod 1)}=0$,
which proves irreducibility.

The Markov process is also aperiodic since $\mathbb{P}\left(\eta_{m+1}=0\,\bigl|\,\eta_{m}=0\right)>0$. 
\end{proof}
From Lemma \ref{lem: markov} it follows that the process $\eta_{m}$
has a unique stationary measure $p$. 
\begin{thm}
\label{thm:when all differences are rational} 
 Assume that $\alpha_{m}$ is equidistributed w.r.t. some measure
$\nu$ on $\mathbb{T}$. Then for $\mu_{K}$-a.e. $x\in K$, the orbit
$\left(D^{m}\left(x\right)\right)_{m=0}^{\infty}$ is equidistributed
w.r.t. the measure $\nu*p*\widetilde{\mu}_{K}$ (where $\widetilde{\mu}_{K}$
is the projection of $\mu_{K}$ to $\mathbb{T}$). 
\end{thm}

In order to prove Theorem \ref{thm:when all differences are rational},
we will need the following property of aperiodic, irreducible Markov
chains. The proof of the following proposition uses some of the ideas
in the proof of Proposition \ref{prop:Simmons-Weiss-equi-of-pairs},
given in \cite{SimmonsWeiss2019}. 
\begin{prop}
\label{prop:subsequence of a Markov chain}Let $x=\left(x_{1},x_{2},...\right)\in\Omega^{\mathbb{N}}$
be an aperiodic, irreducible Markov chain with a finite state space
$\Omega$ and a transition matrix $P$. Let $p$ be the unique stationary
measure for the process and let $\mu$ be the corresponding measure
on $\Omega^{\mathbb{N}}$ w.r.t. $p$ as the starting probability
for the process (i.e., $\mu\left(\left\{ \omega\in\Omega^{\mathbb{N}}:\,\omega_{1}\in A\right\} \right)=p\left(A\right)$
for every $A\subseteq\Omega$). Then for every strictly increasing
sequence of positive integers $\left(n_{k}\right)_{k=1}^{\infty}$,
for $\mu$-a.e. $x\in\Omega^{\mathbb{N}}$, the sequence $\left(T^{n_{k}}\left(x\right)\right)_{k=1}^{\infty}$
is equidistributed w.r.t. $\mu$. 
\end{prop}

\begin{proof}
For the rest of the paper, given a finite sequence $\omega=\left(\omega_{1},...,\omega_{\ell}\right)\in\Omega^{\ell}$,
we denote the corresponding cylinder set by $\left[\omega\right]\df\left\{ \xi\in\Omega^{\mathbb{N}}:\,\left(\xi_{1},...,\xi_{\ell}\right)=\left(\omega_{1},...,\omega_{\ell}\right)\right\} $.

Let $\mathcal{B}_{m}$ be the $\sigma$-algebra generated by the first
$m$ coordinates of $\Omega^{\mathbb{N}}$. Given some fixed $\omega=\left(\omega_{1},...,\omega_{\ell}\right)\in\Omega^{\ell}$
for any $\ell\in\mathbb{N}$, define 
\[
\varphi_{k,m,\omega}\df\mathbb{E}\left[\boldsymbol{1}_{\left[\omega\right]}\circ T^{n_{k}}-\mu\left(\left[\omega\right]\right)|\,\mathcal{B}_{m}\right].
\]
Note that $\boldsymbol{1}_{\left[\omega\right]}\left(T^{n_{k}}\left(x\right)\right)$
only depends on the first $n_{k}+l$ coordinates of $x$, and so for
$n_{k}\leq m-\ell$, 
\[
\varphi_{k,m,\omega}=\boldsymbol{1}_{\left[\omega\right]}\circ T^{n_{k}}-\mu\left(\left[\omega\right]\right).
\]

On the other hand, for $n_{k}>m,$ the random variable $\varphi_{k,m,\omega}$
is constant on atoms of the $\sigma$-algebra $\mathcal{B}_{m}$,
which are just the cylinders $\left[\tau\right]$ for $\tau\in\Omega^{m}$.
For each such cylinder $\left[\tau\right]$ such that $\mu\left(\left[\tau\right]\right)>0$,
by the definition of conditional expectation we have $\forall y\in\left[\tau\right],$
\[
\varphi_{k,m,\omega}\left(y\right)=\dfrac{1}{\mu\left(\left[\sigma\right]\right)}\cdot\mu\left(T^{-n_{k}}\left[\omega\right]\cap\left[\sigma\right]\right)-\mu\left(\left[\omega\right]\right).
\]
Note that $\dfrac{1}{\mu\left(\left[\sigma\right]\right)}\cdot\mu\left(T^{-n_{k}}\left[\omega\right]\cap\left[\sigma\right]\right)$
is nothing but the measure of $T^{-n_{k}}\left[\omega\right]$ conditioned
on the atom $\left[\sigma\right]$, or in a more probabilistic language,
it is
\[
\mathbb{P}\left[x\in T^{-n_{k}}\left[\omega\right]\vert x\in\left[\sigma\right]\right].
\]

Since the process is irreducible and aperiodic, there are constants
$C>0$ and $\alpha\in\left(0,1\right)$, such that for $\mu$ almost
every $x\in\Omega^{\mathbb{N}}$, 
\[
\left|\varphi_{k,m,\omega}\left(x\right)\right|\leq C\cdot\alpha^{n_{k}-m}
\]
(see e.g. \cite[Theorem 4.9]{LevinPeresWilmer2009}).

In view of the above, we see that the sum $\sum_{k=1}^{\infty}\varphi_{k,m,\omega}$
almost surely converges, and so we may define the random variables
\[
M_{m}\df\sum_{k=1}^{\infty}\varphi_{k,m,\omega}.
\]

Denote $K_{1}\left(m\right)=\max\left\{ k:\,n_{k}\leq m-l\right\} ,\,K_{2}\left(m\right)=\max\left\{ k:\,n_{k}\leq m\right\} $.
Note that $K_{1}$ and $K_{2}$ are both increasing functions of $m$
(but not strictly increasing in general). Writing $M_{m}$ as 
\[
M_{m}=\sum_{k=1}^{K_{1}\left(m\right)}\left[\boldsymbol{1}_{\left[\omega\right]}\circ T^{n_{k}}-\mu\left(\left[\omega\right]\right)\right]+\sum_{k=K_{1}\left(m\right)+1}^{K_{2}\left(m\right)}\varphi_{k,m,\omega}+\sum_{k=K_{2}\left(m\right)+1}^{\infty}\varphi_{k,m,\omega}
\]
and noting that $\left|\sum\limits _{k=K_{1}+1}^{K_{2}}\varphi_{k,m,\omega}\right|\leq\ell$,
we have 
\[
M_{m}-L\leq\sum_{k=1}^{K_{1}\left(m\right)}\left[\boldsymbol{1}_{\left[\omega\right]}\left(T^{n_{k}}x\right)-\mu\left(\left[\omega\right]\right)\right]\leq M_{m}+L
\]
for some constant $L>0$. Since $M_{m}$ is a martingale w.r.t. the
increasing sequence of $\sigma$-algebras $\left(\mathcal{B}_{m}\right)_{m\in\mathbb{N}}$,
by Doob's martingale convergence theorem, $M_{m}$ converges almost
surely to some number, which implies that for $\mu$-almost every
$x$, 
\[
\dfrac{1}{m}\sum_{k=1}^{m}\left[\boldsymbol{1}_{\left[\omega\right]}\left(T^{n_{k}}x\right)-\mu\left(\left[\omega\right]\right)\right]\underset{m\to\infty}{\longrightarrow}0.
\]

Since the countable family of cylinder sets generates the Borel $\sigma$-algebra
of subsets of $\Omega^{\mathbb{N}}$, we get that for $\mu$-almost
every $x$, for every function $f\in C\left(\Omega^{\mathbb{N}}\right)$,
\[
\dfrac{1}{m}\sum_{k=1}^{m}f\left(T^{n_{k}}x\right)\underset{m\to\infty}{\longrightarrow}\int fd\mu.
\]
\end{proof}
\begin{cor}
\label{cor:product of equi sequences}Let $\gamma_{m}$ be an equidistributed
sequence w.r.t. some Borel probability measure $\sigma$ on a compact
metric space $X$, and let $\left(\Omega^{\mathbb{N}},T,\mu\right)$
be as above. Then for $\mu$-a.e. $i\in\Omega^{\mathbb{N}}$, the
sequence $\left(\gamma_{m},T^{m}i\right)_{m=1}^{\infty}$ is equidistributed
w.r.t. $\sigma\otimes\mu$. 
\end{cor}

\begin{proof}
Recall that a Borel set $U\subset X$ is called a continuity set with
respect to the measure $\sigma$, if $\sigma(\partial U)=0$. As a
compact metric space, $X$ has a countable basis consisting only of
continuity sets (see \cite[Chapter 3.2]{KuipersNiederreiter2012}),
which we denote by $\mathcal{U}$. Given a set $I\times\left[\omega\right]\subseteq X\times\Omega^{\mathbb{N}}$,
where $I$ is a continuity set w.r.t. $\sigma$, and $\omega\in\Omega^{\ell}$
for some $\ell\in\mathbb{N}$, consider the sum 
\[
\frac{1}{N}\sum_{m=1}^{N}\boldsymbol{1}_{I\times\left[\omega\right]}\left(\gamma_{m},T^{m}i\right).
\]
Define $A=\left\{ m\in\mathbb{N}:\,\gamma_{m}\in I\right\} $ and
let $\left(m_{k}\right)_{k\in\mathbb{N}}$ be an increasing enumeration
of all elements in $A$. Since $I$ is a continuity set w.r.t. $\sigma$,
by the equidistribution of $\gamma_{m}$ we know that 
\[
\lim_{N\to\infty}\frac{\left|A\cap\left\{ 1,...,N\right\} \right|}{N}=\sigma\left(I\right).
\]
By Proposition \ref{prop:subsequence of a Markov chain}, $\left(T^{m_{k}}i\right)_{k=1}^{\infty}$
is a.s. equidistributed w.r.t. $\mu$, hence a.s. 
\[
\frac{1}{k}\sum_{j=1}^{k}\boldsymbol{1}_{\left[\omega\right]}\left(T^{m_{j}}i\right)\longrightarrow\mu\left(\left[\omega\right]\right).
\]
Denote $B=B\left(i\right)=\left\{ m\in A:\,T^{m}i\in\left[\omega\right]\right\} $.
Then 
\[
\lim_{N\to\infty}\frac{\left|B\cap\left\{ 1,...,N\right\} \right|}{\left|A\cap\left\{ 1,...,N\right\} \right|}=\mu\left(\left[\omega\right]\right)\text{ a.s.}
\]
Hence, 

\[
\begin{split} & {\displaystyle \frac{1}{N}\sum_{m=1}^{N}\boldsymbol{1}_{I\times\left[\omega\right]}\left(\gamma_{m},T^{m}i\right)}={\displaystyle \dfrac{\left|B\cap\left\{ 1,...,N\right\} \right|}{N}}\\
 & {\displaystyle =\dfrac{\left|A\cap\left\{ 1,...,N\right\} \right|}{N}\cdot\dfrac{\left|B\cap\left\{ 1,...,N\right\} \right|}{\left|A\cap\left\{ 1,...,N\right\} \right|}\underset{N\to\infty}{\longrightarrow}\text{\ensuremath{\sigma}}\left(I\right)\cdot\mu\left(\left[\omega\right]\right)}
\end{split}
\]
We have obtained that for every continuity set $I$, and every $\omega\in\bigcup\limits _{l\in\mathbb{N}}\Omega^{\ell}$,
the following holds for $\mu$-a.e. $i\in\Omega^{\mathbb{N}}$:
\begin{equation}
\frac{1}{N}\sum_{m=1}^{N}\boldsymbol{1}_{I\times\left[\omega\right]}\left(\gamma_{m},T^{m}i\right)\underset{N\to\infty}{\longrightarrow}\text{\ensuremath{\sigma}}\left(I\right)\cdot\mu\left(\left[\omega\right]\right)=\int\boldsymbol{1}_{I\times\left[\omega\right]}d(\sigma\otimes\mu).\label{eq:equidistribution of a joint sequence}
\end{equation}
Let $\mathcal{A}$ be the algebra generated by $\mathcal{U}$. $\mathcal{A}$
is countable, and it may be easily verified that all the elements
of $\mathcal{A}$ are continuity sets as well. Therefore, for $\mu$-a.e.
$i$, \eqref{eq:equidistribution of a joint sequence} holds for every
$I\in\mathcal{A}$, and $\omega\in\bigcup\limits _{l\in\mathbb{N}}\Omega^{\ell}$.
The claim now follows from \cite[Theorem 1.9]{prokhorov1956convergence}%
.
\end{proof}


\begin{proof}[Proof of Theorem \ref{thm:when all differences are rational}]
By Corollary \ref{cor:product of equi sequences}, for $\beta$-a.e.
$i\in\Lambda^{\mathbb{N}}$, the sequence $\left(\alpha_{m},\eta_{m}\left(i\right)\right)$
is equidistributed w.r.t. $\nu\otimes p$, which implies that the sequence
$\alpha_{m}+\eta_{m}\left(i\right)$ is equidistributed
w.r.t. $\nu*p$. Using Proposition \ref{prop:Simmons-Weiss-equi-of-pairs}
exactly in the same way as in the proof of Theorem \ref{thm:normal number is fractals},
we may deduce that for $\beta$-a.e. $i\in\Lambda^{\mathbb{N}}$,
$\alpha_{m}+\eta_{m}\left(i\right)+\pi(T^{m}(i))$
is equidistributed w.r.t. the measure $\nu*p*\widetilde{\mu_{K}}$. 
\end{proof}

From Theorem \ref{thm:when all differences are rational}
we readily obtain:

\begin{proof}[Proof of Theorem \ref{thm:translations are normal to base D.}]
Suppose $t_{1}$ is normal to base $D$. Then so is $\dfrac{D}{D-1}t_{1}$,
and therefore $\alpha_{m}$ is equidistributed w.r.t. Haar measure.
The conclusion now follows immediately from Theorem \ref{thm:when all differences are rational}. 
\end{proof}

It is not hard to find an example in which the $t_i$ are irrational
and not normal to
base $D$, and the conclusion of Theorem \ref{thm:translations are
  normal to base D.} fails. For example, this will hold when the $t_i$
have many appearances of long strings of the digit 0 in base $D$.
Nevertheless, 
the converse to Theorem \ref{thm:translations are normal to base D.}
is also false. 
Here is a counterexample.
\begin{example}
Denote 
\[
f_{1}^{\alpha}\left(x\right) \df\frac{1}{4}x+\alpha,\ \ \
f_{2}^{\alpha}\left(x\right)\df \frac{1}{4}x+\frac{1}{2}+\alpha.
\]
Let $K_{\alpha}$ be the attractor of the IFS $\left\{ f_{1}^{\alpha},f_{2}^{\alpha}\right\} $
for a given value of $\alpha$. Note that changing $\alpha$ corresponds
to translating the fractal $K_{0}$. More precisely, $K_{\alpha}=K_{0}+c_{\alpha}$
where $c_{\alpha}=\dfrac{4}{3}\alpha$. Let $\mu_{\alpha}$ be the
$\left(\frac{1}{2},\frac{1}{2}\right)$-Bernoulli measure on $K_{\alpha}$.
Note that for all  $ n\in\mathbb{Z}$, 
\[
\hat{\mu}_{\alpha}\left(n\right)=e^{2\pi inc_{\alpha}}\hat{\mu}_{0}\left(n\right),
\]
hence $\hat{\mu}_{\alpha}\left(n\right)=0$ if and only if
$\hat{\mu}_{0}\left(n\right)=0$. 

Denoting $\Delta_{1}=0,\,\Delta_{2}=\frac{1}{2}$ and $\Lambda=\left\{ 1,2\right\} $,
the Fourier transform of $\mu_{0}$ may be calculated as follows (see
\cite[proof of Theorem 6.1]{bugeaud2012distribution}): 
\[
\begin{split}
\hat{\mu}_{0}\left(n\right)= & 
{\displaystyle \lim\limits _{N\to\infty}2^{-N}\sum\limits _{j\in\Lambda^{N}}\exp\left(2\pi in\sum_{s=1}^{N}4^{-s+1}\Delta_{j_{s}}\right)}\\
= & \lim\limits _{N\to\infty}2^{-N}\prod\limits
_{s=0}^{N-1}\left(1+\exp\left(2\pi in4^{-s}\frac{1}{2}\right)\right) .
\end{split}
\]
Therefore
\begin{equation}
\left|\hat{\mu}_{0}\left(n\right)\right|=\prod_{s=0}^{\infty}\left|\cos\left(4^{-s}\frac{1}{2}\pi n\right)\right|.\label{eq:Fourier transform as product}
\end{equation}
Using equation \ref{eq:Fourier transform as product}, we see that for
all $k, m \in \mathbb{Z}$ with $k\geq0$, we have
$\hat{\mu}_{0}\left(4^{k}\left(2m+1\right)\right)=0$. 

Let $\nu$ be the $\left(\frac{1}{2},\frac{1}{2}\right)$-Bernoulli
measure defined on the attractor of the IFS $\left\{ x\mapsto\frac{1}{4}x,\,x\mapsto\frac{1}{4}x+\frac{1}{4}\right\} $.
Analyzing $\hat{\nu}$ in the same way we analyzed $\hat{\mu}_{0}$,
we see that 
\[
\left|\hat{\nu}\left(n\right)\right|=\prod_{s=0}^{\infty}\left|\cos\left(4^{-s}\frac{1}{4}\pi n\right)\right|,
\]
and hence for all $k,m\in\mathbb{Z}$ with $k\geq0$, we have
$\hat{\nu}\left(4^{k}2\left(2m+1\right)\right)=0$. 

Since $\nu$ is ergodic for the $\times4$ map, it
has generic points. Let $t$ be a generic point for $\nu$, and
$\tilde{t} \df \frac{3}{4}t$. By equation \eqref{eq: D^nx}, we see that
for every $x\in K_{\tilde{t}}$, if $x=\lim\limits _{n\to\infty}f_{i_{1}}^{\tilde{t}}\circ\cdots\circ f_{i_{n}}^{\tilde{t}}\left(0\right)$
then for every $n\in\mathbb{N}$, 
\[
  4^{n}x 
  = \sum\limits
  _{j=1}^{n}4^{j}\tilde{t}+\pi\left(T^{n}\left(i\right)\right) 
  = 
4^{n}\frac{4}{3}\tilde{t}-\frac{4}{3}\tilde{t}+\pi\left(T^{n}\left(i\right)\right)
 = 
4^{n}t+\pi\left(T^{n}\left(i\right)\right)-t,
\]
where $\pi$ is the coding map for the IFS $\left\{ f_{1}^{\tilde{t}},f_{2}^{\tilde{t}}\right\} $.
By Corollary \ref{cor:product of equi sequences} and the computation
above, we get that for $\mu_{\tilde{t}}$-a.e. $x\in K_{\tilde{t}}$,
the orbit $\left(4^{n}x\right)_{n=1}^{\infty}$ is
equidistributed w.r.t. the measure $\nu*\mu_{\tilde{t}}$ translated
by $-t$. 
\begin{claim*} For all nonzero 
$w\in\mathbb{Z}$, there exist $k\in\mathbb{N}\cup\left\{ 0\right\} $
and 
$ m\in\mathbb{Z}$ such that either  $w=4^{k}\left(2m+1\right)$ or
$w=4^{k}\left(4m+2\right)$.  
\end{claim*}
\begin{proof}[Proof of Claim]
Clearly, it's enough to prove the statement for the case $4\nmid w$.
If $w$ is odd, then $w=4^{0}\left(2m+1\right)$ for some $m\in\mathbb{Z}$.
Otherwise, $w$ is even, and since we assume $4\nmid w$, then
$w=4^{0}\left(4m+2\right)$ 
for some $m\in\mathbb{Z}$. 
\end{proof}
By the Claim and the analysis of the Fourier transforms of $\nu$
and $\mu_{0}$ given above, we get that for every $0\neq n\in\mathbb{Z}$,
\[
\widehat{\nu*\mu}_{\tilde{t}}\left(n\right)=
\hat{\nu}\left(n\right)\cdot\hat{\mu}_{\tilde{t}}\left(n\right)=0.  
\]
This implies that $\nu*\mu_{\tilde{t}}$ is Haar measure on $\mathbb{T}$,
and ultimately we get that for $\mu_{\tilde{t}}$-a.e. $x\in K_{\tilde{t}}$,
the orbit $\left(4^{n}x\right)_{n=1}^{\infty}$ is
equidistributed w.r.t. Haar measure on $\mathbb{T}$ although $\tilde{t}$
is not normal to base $4$. 
\end{example}

\begin{rem}
The convolution in the example above may also be viewed as
follows. The probability measure $\mu_{0}$ is the law of the random variable
$\sum_{j=1}^{\infty}4^{-j}\xi_{j}$, 
where the $\xi_{j}$ are IID variables which assume the values $0,\,2$
with probability $\frac{1}{2}$. The measure $\nu$ is the law of the random
variable $\sum_{j=1}^{\infty}4^{-j}\chi_{j}$, where the $\chi_{j}$
are IID variables which assume the values $0,\,1$ with probability
$\frac{1}{2}$. Hence, $\nu*\mu_{0}$ is the law of the random variable 
$\sum_{j=1}^{\infty}4^{-j}\chi_{j}+\sum_{j=1}^{\infty}4^{-j} \xi_{j}$.
But 
\[
\sum_{j=1}^{\infty}4^{-j}\chi_{j}+\sum_{j=1}^{\infty}4^{-j}\xi_{j}=\sum_{j=1}^{\infty}4^{-j}\left(\chi_{j}+\xi_{j}\right)
\]
and since $\chi_{j}+\xi_{j}$ are IID random variables that take the
values $0,1,2,3$ with probability $\frac{1}{4}$ each, $\nu*\mu_{0}$
is actually Haar measure on $\mathbb{T}$. Therefore, $\nu*\mu_{\tilde{t}}$
is also Haar measure. 
\end{rem}

 \bibliographystyle{abbrv}
\bibliography{all}

\end{document}